\documentclass[12pt]{article}
\usepackage[textwidth=16cm,textheight=22.5cm]{geometry}
\usepackage{graphicx}
\usepackage{multirow}
\usepackage{amsmath,amssymb,amsfonts}
\usepackage{amsthm}
\usepackage[margin=10pt,font=small]{caption}
\usepackage{mathrsfs}
\usepackage[title]{appendix}
\usepackage{xcolor}
\usepackage{textcomp}
\usepackage{manyfoot}
\usepackage{booktabs}
\usepackage{algorithm}
\usepackage{algorithmicx}
\usepackage{algpseudocode}
\usepackage{listings}
\usepackage{authblk}
\usepackage[square,numbers]{natbib}
\usepackage{hyperref}

\newtheorem{theorem}{Theorem}[section]
\newtheorem{proposition}[theorem]{Proposition}%
\newtheorem{remark}[theorem]{Remark}%
\newtheorem{definition}[theorem]{Definition}%

\providecommand{\keywords}[1]
{	
 {\small \textbf{\textit{Keywords---}} #1}
}

\title{Invariant manifolds in a reversible Hamiltonian system: the tentacle-like geometry}

\author[1]{P. S. Casas\thanks{pabloscasas@uniovi.es}}
\author[1]{F. Drubi\thanks{drubifatima@uniovi.es}}
\author[1]{S. Ib\'a\~{n}ez\thanks{Corresponding author: mesa@uniovi.es}}
\affil[1]{Department of Mathematics, University of Oviedo, Oviedo, Spain}
\date{\today}

\begin{document}
\maketitle
\begin{abstract}
We study a one-parameter family of time-reversible Hamiltonian vector fields in $\mathbb{R}^4$, which has received great attention in the literature. On the one hand, it is due to the role it plays in the context of certain applications in the field of Physics or Engineering and, on the other hand, we especially highlight its relevance within the framework of generic unfoldings of the four-dimensional nilpotent singularity of codimension four.

The system exhibits a bifocal equilibrium point for a range of parameter values. The associated two-dimensional invariant manifolds, stable and unstable, fold into the phase space in such a way that they produce intricate patterns. This entangled geometry has previously been called \textit{tentacular geometry}.
We consider a three-dimensional level set containing the bifocal equilibrium point to gain insight into the folding behavior of these invariant manifolds. Our method consists of describing the traces left by invariant manifolds when crossing an invariant cross section by the reversibility map. With this new approach, we provide a better understanding of how tentacular geometry evolves with respect to the parameter. 

Our techniques enables us to link the tentacular geometry on the cross section with the study of cocooning cascades of homoclinic tangencies. Indeed, we present a general theory to extend to $\mathbb{R}^4$ the phenomena related to cocoon bifurcation that classically develop within families of three-dimensional reversible vector fields. On the basis of these results, we conjecture the existence of heteroclinic cycles consisting of two orbits connecting the bifocus with a saddle node periodic orbit. 
\end{abstract}

\section{Introduction}\label{sec1}
The one-parameter family of four-dimensional vector fields
\begin{equation}
  \left\{\begin{array}{l}
    x_1' = x_2,\\
    x_2' = x_3,\\
    x_3' = x_4,\\
    x_4' = - x_1 + \eta_3 x_3 + x_1^2,
  \end{array}\right.
  \label{sistema_principal}
\end{equation}
with $\eta_3 \in \mathbb{R}$, has been widely studied in the literature \cite{amitol1992,baribarod2011,baribarod2016,buf1996,bufchatol1996,chatol1993}. It is well known that the system exhibits a bifocal equilibrium point at the origin $\mathcal{O}$ when $\eta_3 \in (-2,2)$, that is, the linearization at $\mathcal{O}$ has two pairs of imaginary eigenvalues with real parts of opposite signs. Our goal is to provide novel insights into how the two-dimensional invariant manifolds at $\mathcal{O}$, the unstable $W^u(\mathcal{O})$ and the stable $W^s(\mathcal{O})$, intersect in phase space. In particular, we improve our understanding of the mechanisms that lead to the formation and disappearance of bifocal homoclinic orbits.

The family \eqref{sistema_principal} can be written as a fourth-order scalar differential equation,
\begin{equation}
\label{4thorder_equation}
u^{(iv)}+Pu''+u-u^2=0,
\end{equation}
with $\eta_3=-P$. This equation arises in at least two different applications, the buckling of a strut on a nonlinear elastic foundation \cite{hunboltho1989,hunwad1991}, and as an approximation to traveling solitary waves in the presence of surface tension \cite{amikir1989,grimalben1994,iooper1993}. A brief description of related models can be found in \cite{bufchatol1996}. Moreover, the system \eqref{sistema_principal} can be obtained as a subsystem of the limit family associated with generic unfoldings of the four-dimensional nilpotent singularity of codimension four. To bear this out, as argued in \cite{druibarod2007}, we recall that any generic unfolding of a four-dimensional nilpotent singularity of codimension four can be written, after reduction to normal formal and considering convenient parameters, as
\begin{equation*}
  \left\{\begin{array}{l}
    x_1' = x_2,\\
    x_2' = x_3,\\
    x_3' = x_4,\\
    x_4' = \mu_1 + \mu_2 x_2 + \mu_3 x_3 + \mu_4 x_4 + x_1^2 + h(x,\mu),
  \end{array}\right.
\end{equation*}
where $h$ verifies $h(0,\boldsymbol{\mu})=0$, $(\partial h/\partial x_i)(0,\boldsymbol{\mu})=0$ for $i=1,\ldots,4$, $(\partial^2 h/\partial x^2_1)(0,\boldsymbol{\mu})=0$, $h(\boldsymbol{x},\boldsymbol{\mu})=O(\|(\boldsymbol{x},\boldsymbol{\mu})\|^2)$, and $h(\boldsymbol{x},\boldsymbol{\mu})=O(\|(x_2,x_3,x_4)\|)$, with $\boldsymbol{x}=(x_1,x_2,x_3,x_4)$ and $\boldsymbol{\mu}=(\mu_1,\mu_2,\mu_3,\mu_4)$. 

We can rescale variables and parameters to transform the above system into
\begin{equation*}
  \left\{\begin{array}{l}
    \bar x_1' = \bar x_2,\\
    \bar x_2' = \bar x_3,\\
    \bar x_3' = \bar x_4,\\
    \bar x_4' = \bar\nu_1 + \bar\nu_2 \bar x_2 + \bar \nu_3 \bar x_3 + \bar \nu_4 \bar x_4 + \bar x_1^2 + O(\varepsilon),
  \end{array}\right.
\end{equation*}
with $\bar\nu_1^2+\bar\nu_2^2+\bar\nu_3^2+\bar\nu_4^2=1$, $\varepsilon>0$ and $(\bar x_1,\bar x_2,\bar x_3,\bar x_4) \in A$, a fixed ball in $\mathbb{R}^4$. Understanding the limit family (when $\varepsilon=0$) is clearly a key point in the study of the general unfolding. For this purpose, instead of a spherical blow-up in the parameter space, it is more convenient to apply a directional rescalings. In particular, we consider the above limit family with $\bar\nu_i=\pm 1$ fixed and all others $\bar\nu_j$ varying in $\mathbb{R}$ for $j \neq i$. Finally, we take $\bar\nu_1=-1$ to obtain
\begin{equation*}
  \left\{\begin{array}{l}
    \bar x_1' = \bar x_2,\\
    \bar x_2' = \bar x_3,\\
    \bar x_3' = \bar x_4,\\
    \bar x_4' = -1 + \bar\nu_2 \bar x_2 + \bar \nu_3 \bar x_3 + \bar \nu_4 \bar x_4 + \bar x_1^2,
  \end{array}\right.
\end{equation*}
with $(\bar\nu_2,\bar\nu_3,\bar\nu_4)\in\mathbb{R}^3$. The above system has two equilibrium points $(\pm 1,0,0,0)$. Translating $(-1,0,0,0)$ to the origin and, after a final rescaling of variables and parameters, we get
\begin{equation}
\label{eq:limit_family}
  \left\{\begin{array}{l}
    x_1' = x_2,\\
    x_2' = x_3,\\
    x_3' = x_4,\\
    x_4' = -x_1 + \eta_2 x_2 + \eta_3 x_3 + \eta_4 x_4 + x_1^2.
  \end{array}
  \right.
\end{equation}
Although the whole parameter space is important for the study of the unfolding of the four-dimensional nilpotent singularity of codimension four, the subfamily  \eqref{sistema_principal} obtained for $\eta_2=\eta_4=0$, has an extra relevance as a result of its properties. 
\begin{remark}
    Precisely because we want to keep in mind the role of \eqref{sistema_principal} as a subsystem within family \eqref{eq:limit_family}, we retain the subindex in the parameter $\eta_3$. The terms affected by the parameters $\eta_2$ and $\eta_4$ are only relevant in the context of the unfolding of the four-dimensional nilpotent singularity of codimension four.
\end{remark}

Some basic properties of system \eqref{sistema_principal} are summarized below:
\begin{enumerate}
\item[(\textbf{P1})] The family is time-reversible, namely, it is invariant under the involution
\begin{equation}
R:(x_1,x_2,x_3,x_4)\longrightarrow(x_1,-x_2,x_3,-x_4)
\label{reversor}
\end{equation}
and the time reverse $t \to -t$.
\item[(\textbf{P2})] It has the first integral
\begin{equation}
H(x_1,x_2,x_3,x_4)=\frac{1}{2}x_1^2-\frac{1}{3}x_1^3-\frac{\eta_3}{2}x_2^2+x_2x_4-\frac{1}{2}x_3^2.
\label{hamiltonian}
\end{equation}
\item[(\textbf{P3})] There exist two equilibrium points, one at the origin $\mathcal O=(0,0,0,0)$ and the other at $(1,0,0,0)$. Moreover, equilibria belong to different level sets of the Hamiltonian function $H$.
\item[(\textbf{P4})] Regarding the linear part at $\mathcal O$, it can be checked that it has:
\begin{itemize}
    \item four pure imaginary eigenvalues $\pm\, \omega_1 i$ and $\pm\, \omega_2 i$ when $\eta_3 < -2$, with
    \[
    \omega_1=\sqrt{(-\eta_3+\sqrt{\eta_3^2-4})/2}, \qquad
    \omega_2=\sqrt{(-\eta_3-\sqrt{\eta_3^2-4})/2},
    \]
    \item two double pure imaginary eigenvalues $\pm\, i$ when $\eta_3 = -2$,
    \item four complex eigenvalues $\rho \pm \omega i$ and $-\rho \pm \omega i$ when $-2 < \eta_3 < 2$, with
    \[
    \rho=\sqrt{\eta_3+2}/2, \qquad
    \omega=\sqrt{-\eta_3+2}/2,
    \]
    \item two double real eigenvalues $\pm\, 1$ when $\eta_3 = 2$,
    \item four non-zero real eigenvalues $\pm\, \lambda_1$ and $\pm\, \lambda_2$ when $\eta_3 > 2$, with
        \[
    \lambda_1=\sqrt{(\eta_3+\sqrt{\eta_3^2-4})/2}, \qquad
    \lambda_2=\sqrt{(\eta_3-\sqrt{\eta_3^2-4})/2}.
    \]
\end{itemize} 
\end{enumerate}
Since we are interested in the behavior exhibited by the invariant manifolds at $\mathcal O$, and $H(1,0,0,0)$ is not equal to $H(\mathcal{O})$, the point $(1,0,0,0)$ will play no role in all the discussions to come. We only point out that the linear part at $(1,0,0,0)$ always has a pair of real eigenvalues and a pair of complex eigenvalues with a nonzero real part.

Motivated by the already mentioned role that the system \eqref{sistema_principal} has in certain applications, the existence of homoclinic orbits has been extensively studied in the literature. Applying the results given in \cite{hoftol1984}, the authors in \cite{amitol1992} proved that there is a unique and $R$-invariant intersection between $W^u(\mathcal{O})$ and $W^s(\mathcal{O})$ for all $\eta_3 \geq 2$. The transversality of this intersection within the level set $H^{-1}(0)$ was argued in \cite{bufchatol1996}. In this last paper, it was also proven that, when $\eta_3=2$, the results in \cite{chatol1993} can be applied to conclude that there is a Belyakov-Devaney bifurcation \cite{belshi1990,dev1976ham}. As a consequence, it follows the existence of infinitely many $n$-modal secondary homoclinic orbits for all integers $n\geq 2$ and $\eta_3$ close enough to $2$; these orbits cross $n$ times a transversal section to the primary homoclinic orbit. Heuristic arguments in \cite{bufchatol1996}, supported by numerical results, show that the non-degenerate $n$-modal homoclinic orbits arising at $\eta_3=2$ gradually disappear through a cascade of tangent bifurcations as $\eta_3$ decreases from $2$ to $-2$. Part of these bifurcations were theoretically analyzed in \cite{kno1997}. From \cite{iooper1993}, it follows that there are at least two symmetric homoclinic orbits for $\eta_3>-2$ and close enough to $-2$. Furthermore, it is known \cite{buf1995} that there is at least one homoclinic orbit for all $\eta_3 \in (-2,2)$. As pointed out in \cite{bufchatol1996}, when the intersection between the two-dimensional invariant manifolds $W^u(\mathcal{O})$ and $W^s(\mathcal{O})$ is transverse along such a homoclinic orbit, the results in \cite{dev1976ham} regarding the existence of infinitely many homoclinics and horseshoes in $H^{-1}(0)$ apply to all $\eta_3 \in (-2,2)$. Other interesting references can be found in \cite{buf1996}.

The entire discussion about the existence of homoclinic orbits for parameter values in the interval $(-2,2)$ corresponds to the bifocal case. The non-resonant bifocal homoclinic orbits, where contractivity and expansivity are different, were already studied by Shilnikov \cite{shi1967}, who proved the existence of a countable set of periodic orbits. The study of the general case continued with the subsequent articles \cite{fowspa1991,ibarod2015,laigle1997,shi1970,wig2003}. As we have already mentioned, Devaney \cite{dev1976ham} considered the Hamiltonian case, proving that in any section which is transverse to the homoclinic orbit, and within the level set that contains it, there exists a compact hyperbolic invariant set where the dynamics is conjugated with a Bernoulli shift with $N$ symbols. Lerman extended this result in \cite{ler1991,ler1997,ler2000} to prove that, in the level sets close to $H^{-1}(0)$ and over the appropriate cross section, there also exist an infinite number of two-dimensional horseshoes. With an analogous perspective, it was proven in \cite{baribarod2016}  that, in the neighborhood of the bifocal homoclinic orbit, there are invariant compact sets for the return map where the dynamics is conjugated to a Bernoulli shift times the identity.

The similarity between the Hamiltonian and reversible cases has been highlighted on many occasions in the literature (for instance, \cite{dev1976rev,lamrob1998}). In \cite{dev1977}, Devaney proved that both homoclinic and reversible bifocal homoclinic orbits are approximated by a one-parameter family of periodic orbits. Later, H\"arterich \cite{har1998} showed that, in the reversible case, in the neighborhood of the primary homoclinic orbit and for each $n\geq 2$ there are infinitely many $n$-homoclinic orbits and each of them is approximated by one or more families of periodic orbits. In \cite{homlam2006}, the authors studied the existence of horseshoes under the hypothesis that one of the periodic orbits in the family which approximates the connection, has its own associated homoclinic orbit. Without the need for this hypothesis, it was proven in \cite{barrairod2019} that, in the neighborhood of a non-degenerate bifocal homoclinic orbit, for every $n\geq 2$, there is a return map with an invariant set where the dynamics is conjugated to a horseshoe with $n$-symbols.

\begin{remark}
    We refer to \cite{lamrob1998} for a fairly complete survey and an extensive bibliography on reversible dynamical systems.
\end{remark}

All the results we have mentioned above, within the Hamiltonian or reversible context, can be applied to the family \eqref{sistema_principal}. In particular, as argued in \cite{bufchatol1996}, the existence of horseshoes in the neighborhood of a bifocal homoclinic orbit allows us to explain the folding of the invariant manifolds which, in subsequent passes through the equilibrium environment, wrap and unwrap over themselves. The terminology \emph{tentacle-like geometry} is used in \cite[Section 1.1]{bufchatol1996} to refer to the print that the invariant manifolds leave over an appropriate cross section. In this paper, we provide a novel perspective of this tentacle-like geometry, as an alternative tool for studying the genesis and destruction of homoclinic orbits that could be also useful in future discussions about some of the conjectures formulated in \cite{bufchatol1996}. Indeed, in the context of this four-dimensional dynamics, our visualization of the tentacular geometry allows us to place the cocoon bifurcations, that classically appear in reversible three-dimensional systems \cite{dumibakok2006,lau1992}.

We employ a geometric approach to explore the existence of bifocal homoclinic orbits. 
In particular, we study the intersections between the two-dimensional invariant manifolds, $W^u(\mathcal{O})$ and $W^s(\mathcal{O})$, and a three-dimensional cross section containing $\mathrm{Fix}(R)$. The latter is the set of points that remain invariant by the involution $R$, defined in \eqref{reversor}. Moreover, since we are interested in homoclinic orbits to $\mathcal{O}$, and because $H$ (introduced in \eqref{hamiltonian}) is kept constant along orbits and $H(\mathcal{O})=0$, we only need to pay attention to the set of points on the cross section where $H$ is zero. Thus, we can visualize the intersections onto a two-dimensional coordinate system.

Our approach is similar to that used in \cite{lau1992} to study the Michelson system (see \cite{mic1986})
\begin{equation}
  \left\{\begin{array}{l}
    x_1' = x_2,\\
    x_2' = x_3,\\
    x_3' = c^2- \dfrac{x_1^2}{2} - x_2,
  \end{array}\right.
  \label{michelson_system}
\end{equation}
where $c$ is a positive parameter. This family is invariant under the involution
\begin{equation*}
(x_1,x_2,x_3) \longrightarrow (-x_1,x_2,-x_3)
\end{equation*}
and the time reverse $t \to -t$. For all $c>0$, there exist two saddle-focus equilibrium points $P_{\pm}=(\pm \sqrt{2} c,0,0)$ such that $\mathrm{dim}(W^s (P_-))=\mathrm{dim}(W^u (P_+))=2$. With the goal of studying the intersections between these two-dimensional invariant manifolds, Lau considered a cross section containing the set of fixed points of the reversibility and wondered how the invariant manifold intersects it. As we will see later, the geometric structures discovered by Lau are similar to those shown in our study of system \eqref{sistema_principal}. Noticeably, the Michelson system is related to the study of unfolding of the three-dimensional nilpotent singularity of codimension three (see \cite{dumibakok2001,dumibakok2006,dumibakoksim2013}) and it is also relevant to the study of traveling-wave solutions of the Kuramoto-Sivashinsky equation (see \cite{mic1986}). Furthermore, in \cite{carferter2008} a piecewise linear version of the Michelson system is considered.

Finally we describe the structure of the present work. In Section \ref{sec:cross}, we explain how we visualize the geometry of the intersections between the invariant manifolds and an appropriate cross section, as well as the numerical techniques used to approximate such manifolds. We also discuss the properties of these intersections. Our method leads to a labeling of homoclinic orbits, which is also discussed. Section \ref{sec:intersections} covers the evolution of the first four intersections for a set of parameter values. Section \ref{sec:cocoon} shows how, under certain conditions, the results in \cite{dumibakok2006} can be extended to the case of Hamiltonian reversible four-dimensional systems, concluding the existence of cascades of homoclinic tangencies that accumulate on a saddle-node periodic orbit. We also provide numerical evidence of the existence of these cascades in system \eqref{sistema_principal}. To conclude, Section \eqref{sec:discussion} discusses additional aspects and indicates potential future research directions.

\section{Invariant manifolds and cross section}
\label{sec:cross}
As we have mentioned above, our goal is to study some aspects of the evolution of invariant manifolds at the origin. As a consequence of the reversibility, we pay attention only to the unstable invariant manifold $W^u(\mathcal{O})$. Properties of the stable invariant manifold, $W^s(\mathcal{O})$, follow by symmetry.

\subsection{Local invariant manifold}
Linear part at $\mathcal{O}$ is given by the matrix
$$
A=\begin{pmatrix}
0 & 1 & 0 & 0
\\
0 & 0 & 1 & 0
\\
0 & 0 & 0 & 1
\\
-1 & 0 & \eta_3 & 0
\end{pmatrix}.
$$
If we assume $\lambda$ is an eigenvalue of $A$, the eigenvectors can be written as
$$
(x_1,\lambda\, x_1,\lambda^2\, x_1,\lambda^3\, x_1).
$$
Since we are only interested in the case of complex eigenvalues $\lambda = \rho \pm i \omega$, with a nonzero real part, it is easy to deduce that
$$
v=(1,\rho+\omega i,\rho^2 - \omega^2 + 2\, \rho\,\omega i,\rho\,(\rho^2 - 3\, \omega ^2) + \omega\,(3\, \rho^2 - \omega^2) i)
$$
is a complex eigenvector associated to $\rho + i \omega$. Therefore, a basis for the tangent space to $W^u(\mathcal{O})$ at the origin is given by its real and imaginary parts:
$$
\{
(1, \rho, \rho^2-\omega^2, \rho\,(\rho^2-3\,\omega^2)),
(0, \omega, 2\,\rho\,\omega, \omega\,(3\,\rho^2-\omega^2))
\}.
$$
After a straightforward computation, the equations of the tangent plane are
\begin{equation}
\label{eq:plano_tangente}
x_3= -(\rho^2+\omega^2)\, x_1 + 2\, \rho\, x_2, \qquad
x_4= -2\, \rho\, (\rho^2+\omega^2)\,x_1 +(3\,\rho^2-\omega^2)\, x_2.
\end{equation}
We conclude that the \emph{local} unstable manifold can be written as a graph with respect to $(x_1,x_2)$, that is,
\begin{equation}
W^u_{loc}(\mathcal{O}) = \left\{ (x_1, x_2, x_3, x_4) \in \mathbb{R}^4 \, : \, x_3 = a (x_1, x_2), \, x_4 = b(x_1, x_2) \right\},
\label{eq:local_unstable}
\end{equation}
where $a$ and $b$ are analytic functions in a neighbourhood of $(0,0)$.

We consider the power series around the origin of both $a$ and $b$ to write
\begin{equation}
  x_3 = \sum_{M = 1}^{\infty} \sum_{i = 0}^M a_{M-i, i} \, x_1^{M - i}\, x_2^i
  \qquad \mbox{and} \qquad
  x_4 = \sum_{M = 1}^{\infty} \sum_{i = 0}^M b_{M-i, i} \, x_1^{M-i}\, x_2^i.
  \label{series_de_Taylor}
\end{equation}
According to the computations above, it follows that
$$
a_{10} = -(\rho^2+\omega^2),      \quad a_{01} = 2 \, \rho, \quad
b_{10}=-2 \, \rho \, (\rho^2+\omega^2), \quad b_{01} = (3\,\rho^2-\omega^2).
$$

As usual, applying the invariance condition, the coefficients $a_{M -s, s}$ and $b_{M-s, s}$ are computed up to a certain (high) order $M$. In this case, we take derivative
\begin{equation*}
\left\{
\begin{aligned}    
x_3' = \frac{\partial a}{\partial x_1} x_1'+ \frac{\partial a}{\partial x_2} x_2',
\\
x_4' = \frac{\partial b}{\partial x_1} x_1'+ \frac{\partial b}{\partial x_2} x_2',
\end{aligned}
\right.
\end{equation*}
to obtain the invariance equations
\begin{equation*}
\left\{
\begin{aligned}    
b(x_1,x_2) =&\frac{\partial a}{\partial x_1}(x_1,x_2)\, x_2+ \frac{\partial a}{\partial x_2}(x_1,x_2)\, b(x_1,x_2),
\\
-x_1+\eta_3\, a(x_1,x_2) + x_1^2 =& \frac{\partial b}{\partial x_1}(x_1,x_2)\, x_2+ \frac{\partial b}{\partial x_2}(x_1,x_2)\, b(x_1,x_2).
\end{aligned}
\right.
\end{equation*}
With standard but lengthy computations, we obtain a $2 (M+1) \times 2 (M+1)$ system in the unknowns $a_{M-s, s}$ and $b_{M-s, s}$, for each $M = 1, 2, \ldots$ and $s = 0, 1, \ldots, M$:
\begin{equation*}
   \begin{array}{l}
     \left\{\begin{array}{ll}
       \displaystyle\sum_{_{k = 1}}^M a_{k - 1, 1} \,a_{M - k + 1, 0} = b_{M, 0} & \\[3ex]
       (M - s + 1)\, a_{M - s + 1, s - 1} + c_{M - s, s} = b_{M - s, s} & s = 1,
       \ldots, M\\[1ex]
       \displaystyle\sum_{_{k = 1}}^M b_{k - 1, 1} \,a_{M - k + 1, 0} = \eta_3\, a_{M, 0} +
       \eta_4\, b_{M, 0} - \delta_{M 1} + \delta_{M 2} & \\[3ex]
       (M - s + 1) \, b_{M - s + 1, s - 1} + d_{M - s, s} = \eta_3\, a_{M - s, s} & s = 1, \ldots, M
     \end{array}\right.\\[2ex]
   \end{array}
\end{equation*}
with
\[
c_{M - s, s} = \sum_{k = 1}^M \sum_{i = l_k}^{L_k} i\,a_{k - i, i}\, a_{M - k + i - s, s + 1 - i} \quad \mbox{and} \quad d_{M - s, s} = \sum_{k = 1}^M \sum_{i = l_k}^{L_k} i\,b_{k - i, i}\, a_{M - k + i - s, s + 1 - i},
\]
where $l_k = \text{max} \{ s + k - M, 1 \}$, $L_k = \text{min} \{ s + 1,
k \}$ and $\delta_{x y} = \delta_{x y z} = 1$ only when $x = y = z$ but $0$ otherwise. Those systems are solved in increasing order for $M = 1, 2,
\ldots$. Once we compute $a_{M - s, s}$ and $b_{M - s, s}$, we can only  evaluate (\ref{series_de_Taylor}) on a certain disk of convergence centered at $(0, 0)$, in the plane $(x_1, x_2)$.
\begin{remark}
    In the numerical approximation, $M$ is adapted for each initial condition
    $$(x_1, x_2, a(x_1, x_2), b(x_1, x_2)).$$
    The series convergence criterion estimates the relative error of the sum, which is approximated by the included highest order ($M$) term in \eqref{series_de_Taylor}, divided by the present value of the sum. The required $M$ increases when $\eta_3$ approaches $- 2$, where $M \lesssim 40$. 
\end{remark}
\begin{remark}
    The method used to approximate invariant manifolds is a particular case of parameterization methods, widely used (their main ideas) since the beginning of the last century, in the search for invariant objects of dynamical systems \cite{Haro_2016}. 
\end{remark}

\subsection{Fundamental domains}
In order to compute the global unstable invariant manifold, we consider a curve $D^u(\mathcal{O}) \subset W^u_{loc}(\mathcal{O})$, homeomorphic to $\mathbb{S}^1$, such that all orbits in $W^u(\mathcal{O})$ cross $D^u(\mathcal{O})$. Thus, the invariant manifold can be obtained by forward numerical integration of initial conditions on $D^u(\mathcal{O})$. Namely, any curve
\begin{equation}
\label{eq:fundamentaldomain}
D^u(\mathcal{O})=\{(x_1,x_2,a(x_1,x_2),b(x_1,x_2)): x_1=r\cos\theta,\, x_2=r\sin\theta,\, r=r^*, \,\theta \in ]0,2\pi]\},
\end{equation}
such that the circle of radius $r^*$ in the plane $(x_1,x_2)$ is contained in the domain of convergence of the power series \eqref{series_de_Taylor}, satisfies that all orbits in $W^u(\mathcal{O})$ intersect $D^u(\mathcal{O})$. Ideally, $D^u(\mathcal{O})$ should be a fundamental domain, that is, any orbit in $W^u(\mathcal{O})$ should cross $D^u(O)$ once and only once. We cannot affirm a priori that $D^u(\mathcal{O})$ is a fundamental domain, but as we explain below, this follows from numerical computations.

We use a Taylor numerical integrator \cite{Jorba_2005} to approximate the trajectories of  \eqref{sistema_principal}. As a check, at each time instant, we verify that the computed solution $x(t)$ satisfies $H(x(t))\approx0$, provided that $H(x(0))\approx0$, where $H$ is the first integral defined in \eqref{hamiltonian}.

\subsection{Cross section}
To get an understanding of the complex geometry of the invariant manifolds, we analyze their intersection with a suitable cross section, namely
\[
S=\{(x_1,x_2,x_3,x_4)\,:\,x_2=0\}.
\]
In the sequel, we consider the coordinates $(x_1,x_3,x_4)$ in $S$. On the other hand, since $W^u(\mathcal{O}) \subset \{H=0\}$, it follows that $W^u(\mathcal{O}) \cap S \subset C$, where 
\[
C=S \cap \{H=0\}=\left\{(x_1,x_2,x_3,x_4)\,:\,x_2=0,\, \frac{1}{2}x_1^2-\frac{1}{3}x_1^3-\frac{1}{2}x_3^2=0\right\}
\]
is a loop cylinder (see Figure \ref{fig:loop_cylinder}).
\begin{figure}[t]
    \centering
    \includegraphics[width=0.7\linewidth]{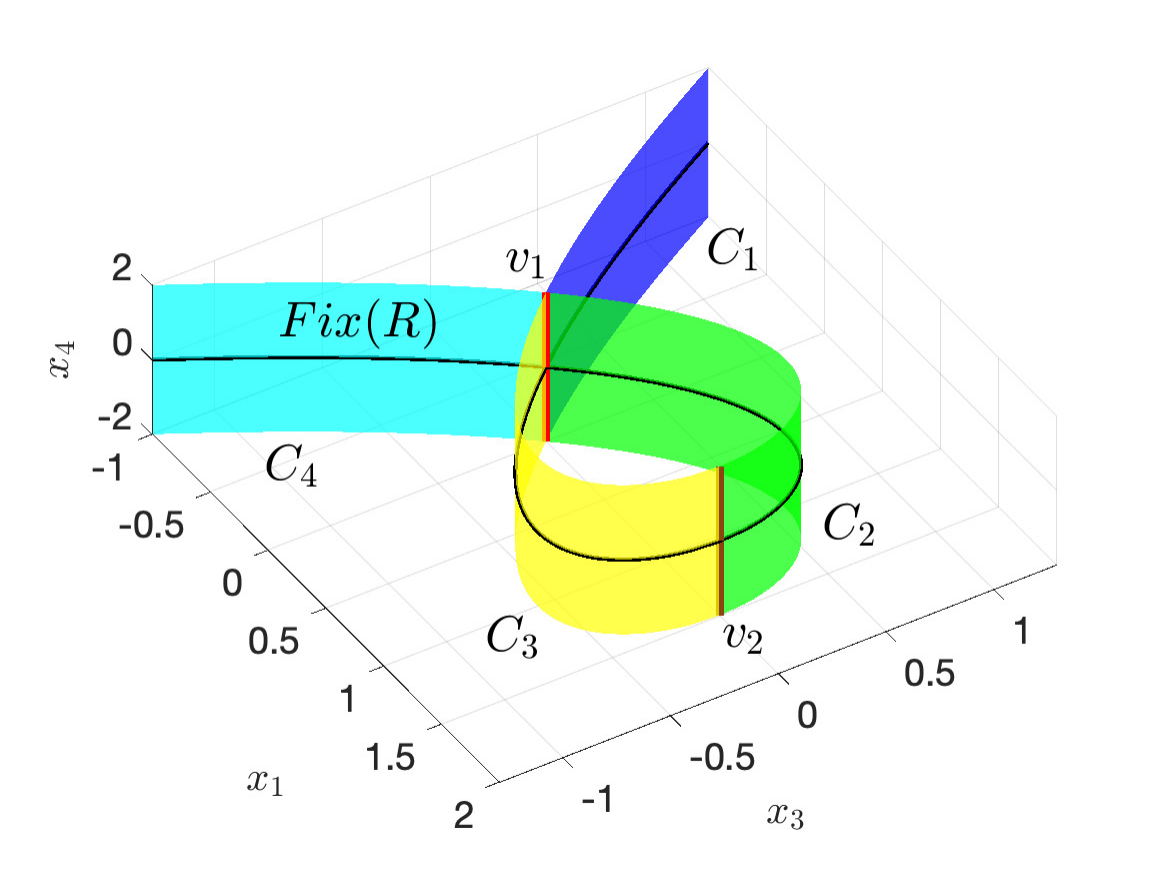}
    \caption{The loop cylinder $C$ is defined as the intersection of the surface $S$ with the level set ${H=0}$. Invariant manifolds intersect the cross section $S$ through the loop cylinder $C$. The vertical lines $v_1$ and $v_2$, where $x'_2=0$, are depicted in red and brown, respectively, and the regions $C_i$, with $i=1,2,3,4$, are colored blue, green, yellow and cyan, respectively. On $C_1$ and $C_2$, $x'_2$ is positive, whereas $x'_2$ is negative on $C_3$ and $C_4$.}
    \label{fig:loop_cylinder}
\end{figure}
On $C$, we distinguish two vertical lines,
\[
\begin{aligned}
    v_1=&\{(x_1,x_3,x_4)\,:\,x_1=x_3=0\},
    \\
    v_2=&\{(x_1,x_3,x_4)\,:\,x_1=1.5,\,x_3=0\},
\end{aligned}
\]
that split $C$ in four regions
\[
\begin{aligned}
    C_1=&\{(x_1,x_3,x_4) \in C \, : \, x_1<0, \, x_3>0\},
    \\
    C_2=&\{(x_1,x_3,x_4) \in C \, : \, x_1>0, \, x_3>0\},
    \\
    C_3=&\{(x_1,x_3,x_4) \in C \, : \, x_1>0, \, x_3<0\},
    \\
    C_4=&\{(x_1,x_3,x_4) \in C \, : \, x_1<0, \, x_3<0\}.
\end{aligned}
\]  
Since $x'_2=x_3$, when an orbit crosses $S$ through $C_1$ or $C_2$, respectively $C_3$ or $C_4$, then $x_2$ is strictly increasing, respectively, decreasing. On the other hand, the orbits through the points in $v_1-\{\mathcal{O}\}$ and $v_2$ are tangent to the section $S$. As explained above, the intersections between the invariant manifolds and $S$ are contained in the loop cylinder $C$. Moreover, due to reversibility, $W^u(\mathcal{O}) \cap S$ and $W^s(\mathcal{O}) \cap S$ are symmetric with respect to $x_4=0$. A three-dimensional representation of $C$, as shown in Figure \ref{fig:loop_cylinder}, does not facilitate the study of these intersections. For this reason, we focus on planar projections of the regions $C_i$. Specifically, we introduce the below shifted projection
\[
P:(x_1,x_3,x_4) \in C \longrightarrow (\bar x_1,x_4) \in \mathbb{R}^2,
\]
where
\[
\bar x_1=\left\{
\begin{aligned}
    x_1 - 1.5 \quad \mbox{if} \quad & x_3 \geq 0,
    \\
    1.5 - x_1 \quad \mbox{if} \quad & x_3 < 0.
\end{aligned}
\right.
\]

\begin{figure}
    \centering
    \includegraphics[width=0.7\linewidth]{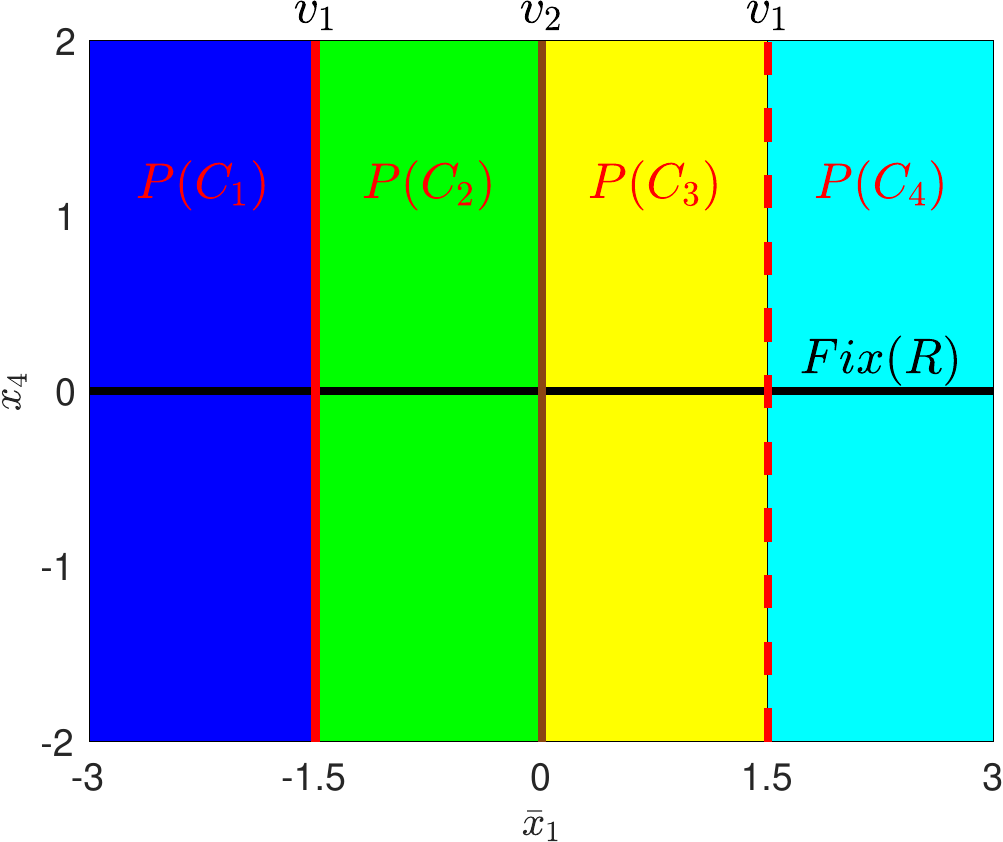}
    \caption{Shifted projection of the loop cylinder $C$ onto a plane. It establishes a bijection between $C$ and the $(\bar x_1, x_4)$ plane, excluding the vertical line $\bar x_1=1.5$.}
    \label{fig:cilindroabierto}
\end{figure}

We notice that $P$ establishes a bijection between $C$ and $\{(\bar x_1, x_4) \in \mathbb{R}^2 : \bar x_1 \neq 1.5\}$. While no point in $C$ projects onto $\bar x_1=1.5$, for the sake of convenience, we can consider that all projections onto the vertical line through $\bar x_1= -1.5$, representing the projection of $v_1$, are duplicated on the vertical line through $\bar x_1=1.5$. In Figure \ref{fig:cilindroabierto}, we show the images of the relevant regions and lines in $C$ under the projection $P$, namely
\[
\begin{aligned}
    P(C_1)&=\{(\bar x_1,x_4) \in  \mathbb{R}^2: \, \bar x_1 < -1.5 \},
    \\
    P(C_2)&=\{(\bar x_1,x_4) \in  \mathbb{R}^2: \, -1.5 < \bar x_1 < 0 \},
    \\
    P(C_3)&=\{(\bar x_1,x_4) \in  \mathbb{R}^2: \, 0 < \bar x_1 < 1.5 \},
    \\
    P(C_4)&=\{(\bar x_1,x_4) \in  \mathbb{R}^2: \, 1.5 < \bar x_1 \},
    \\
    P(v_1)&=\{(\bar x_1,x_4) \in  \mathbb{R}^2: \, \bar x_1=-1.5\},
    \\
    P(v_2)&=\{(\bar x_1,x_4) \in \mathbb{R}^2: \, \bar x_1 = 0\}.
\end{aligned}
\]

\subsection{Intersection between the invariant manifolds and the cross section}
\begin{figure}
    \centering
    \includegraphics[width=0.5\linewidth]{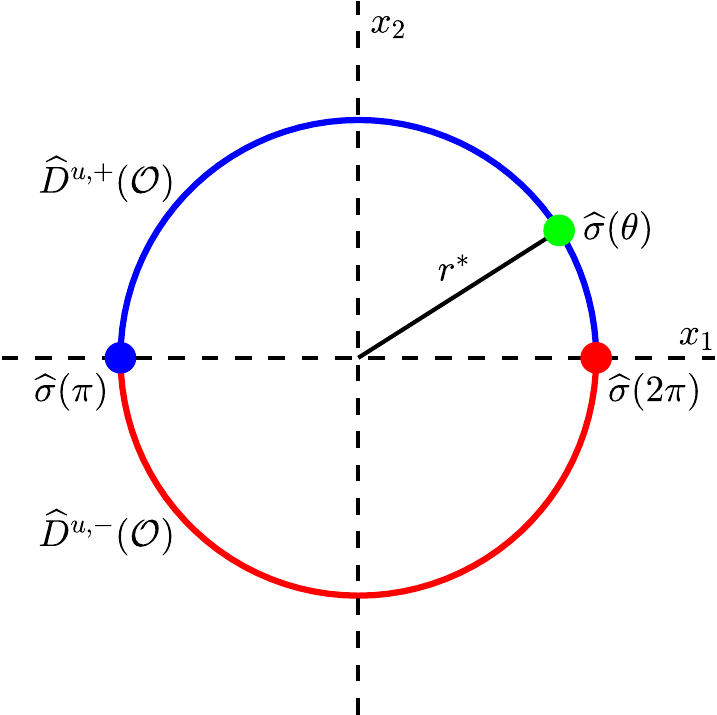}
    \caption{Projection of the fundamental domain $D^u(\mathcal{O})$ on the plane $(x_1,x_2)$. $\widehat{D}^{u,+}$ (in blue) and $\widehat{D}^{u,-}$ (in red) are the projections of $D^{u,+}$ and $D^{u,-}$, respectively. Initial points in $D^{u,+}$, respectively $D^{u,-}$, first intersect $C$ at a point where $x_2$ decreases, respectively increases. The notation $\widehat{\sigma}(\theta)$ refers also to the projection of $\sigma(\theta) \in D^u$.}
    \label{fig:fundamental_domain}
\end{figure}
In order to study the geometry of the intersections between the invariant manifolds and the cross section, we consider a fundamental domain $D^u(\mathcal{O})$ as given in \eqref{eq:fundamentaldomain}. In the sequel, $D^u(\mathcal{O})$ will be parameterized by the angle, that is,
\[
D^u=\{\sigma(\theta):\, \theta\in ]0,2\pi]\},
\]
where
\[
\sigma(\theta)=(x_1,x_2,a(x_1,x_2),b(x_1,x_2)),\quad x_1=r^*\cos\theta,\ x_2=r^*\sin\theta,
\]
with $a$ and $b$ as given in \eqref{eq:local_unstable}.
Furthermore, we consider a partition $D^u=D^{u,+} \cup D^{u,-}$ (see Figure \ref{fig:fundamental_domain}), where
\[
D^{u,+}=\{\sigma(\theta):\, \theta\in ]0,\pi]\} \quad
\mbox{and} \quad D^{u,-}=\{\sigma(\theta):\, \theta\in ]\pi,2\pi]\}.
\]
In the sequel, $\varphi$ denotes the flow of the vector field \eqref{sistema_principal}. We define
\[
\Sigma^u = \bigcup_{\theta\in ]0,2\pi]} \left(\{\varphi(t,\sigma(\theta)): \, t > 0\} \cap S\right) \subset C
\]
and then divide $\Sigma^u$ into separated pieces as follows. First, let be
\[
\Sigma^{u,+}_1 = \{\varphi(t_1(p),p) \in S\,:\,p \in D^{u,+}\} \quad \mbox{and} \quad \Sigma^{u,-}_1 = \{\varphi(t_1(p),p) \in S\,:\,p \in D^{u,-}\},
\]
where $t_1(p)>0$ is such that $\varphi(t_1(p),p) \in S$, but $\varphi(t,p) \notin S$ for all $t \in ]0,t_1(p)[$. Second, for $k \geq 1$, we define
\[
\Sigma^{u,+}_{k+1} = \{\varphi(t_{k+1}(p),p) \in S\,:\,p \in \Sigma^{u,+}_{k}\} \quad \mbox{and} \quad \Sigma^{u,-}_{k+1} = \{\varphi(t_{k+1}(p),p) \in S\,:\,p \in \Sigma^{u,-}_{k}\},
\]
where $t_{k+1}(p)>0$ is such that $\varphi(t_{k+1}(p),p) \in S$, but $\varphi(t,p) \notin S$ for all $t \in ]0,t_{k+1}(p)[ $. Finally, let
\[
\Sigma^{u}_{k} = \Sigma^{u,+}_{k} \cup \Sigma^{u,-}_{k},
\]
for each $k \in \mathbb{N}$.
\begin{figure}[t]
    \centering
    \includegraphics[width=0.7\linewidth]{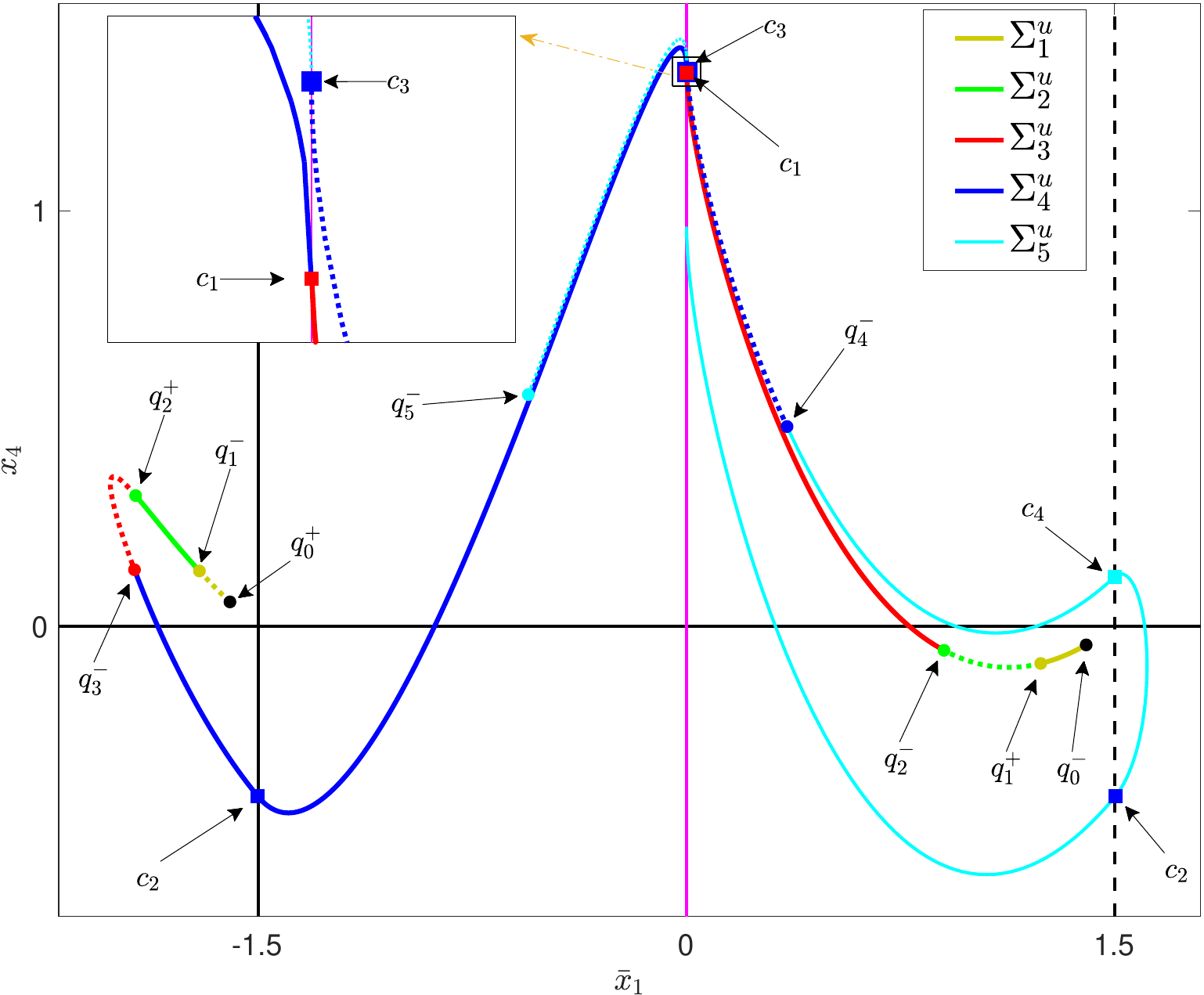}
    \caption{The first five intersections $\Sigma^{u,+}_k$ (solid lines), $\Sigma^{u,-}_k$ (dotted lines), for $k=1,\ldots,5$ and $\eta_3=-1.73$ on $P(C)$. Different colors are used for each $k$. We denote $q_0^+=\sigma(\pi) \in D^{u,+}$ and $q_0^-=\sigma(2\pi) \in D^{u,-}$. Moreover, $q_{k}^\pm=\varphi(t_{k},q_{k-1}^\pm)\in \Sigma_{k}^{u,\pm}$ for each $k\geq 1$. Points $c_i$ for $i=1,2,3$ are related to the propositions \ref{prop:corte_v1_arriba}, \ref{prop:corte_v1_abajo}, \ref{prop:corte_v2_arriba}, and \ref{prop:corte_v2_abajo}.}
    \label{fig:cinco_intersecciones}
\end{figure}
In Figure \ref{fig:cinco_intersecciones}, we show the pieces $\Sigma^{u,+}_k$ (solid lines) and $\Sigma^{u,-}_k$ (dotted lines) for $k=1,\ldots,5$, when $\eta_3=-1.73$.

Similarly, we can define sets $D^s$, $\Sigma^s$ and $\Sigma_k^{s,\pm}$. Nevertheless, because of the reversibility, they are easily obtained by $R(D^u)$, $R(\Sigma^u)$ and $R(\Sigma_k^{u,\pm})$, respectively, where $R$ is given in \eqref{reversor}. Indeed, each piece $\Sigma_k^{s,\pm}$ is the symmetric of $\Sigma_k^{u,\pm}$ with respect to the plane $x_4=0$. 

\begin{proposition}
\label{prop:prop_basicas}
The properties below are satisfied:
\begin{enumerate}
    \item If $r^*$ is small enough, then $\sigma(\pi) \in C_1$ and $\sigma(2\pi) \in C_3$.
    \item If $r^*$ is small enough, then $\varphi(t,\sigma(\pi)) \in \{x_2>0\}$  for $t>0$ small enough, respectively $\varphi(t,\sigma(2\pi)) \in \{x_2<0\}$.
    \item Let $\bar x=(\bar x_1,\bar x_2,\bar x_3,\bar x_4)$, with $\bar x_2>0$ (respectively $\bar x_2<0$),  and assume that there exists $\tau(\bar x)>0$ such that $\bar x^*=\varphi(\tau(\bar x),\bar x) \in S$, but $\varphi(t,\bar x) \notin S$ for all $t \in ]0,\tau(\bar x)[ $, then $\bar x^* \in \overline{C_3 \cup C_4}$ (respectively $\bar x^* \in \overline{C_1 \cup C_2}$).
    \item If $\bar x \in (C_1 \cup C_2) \cap \Sigma_k^{u}$ (respectively $\bar x \in (C_3 \cup C_4) \cap \Sigma_k^{u}$)  and there exists $\bar x^* = \varphi(t_{k+1}(\bar x),\bar x) \in \Sigma_{k+1}^{u}$, then $\bar x^* \in \overline{C_3 \cup C_4}$ (respectively $\bar x^* \in \overline{C_1 \cup C_2}$).
\end{enumerate}
\end{proposition}
\begin{proof}
(1) By definition,
\[
\sigma(\pi)\equiv (\sigma_1(\pi),\sigma_2(\pi),\sigma_3(\pi),\sigma_4 (\pi))=(-r^*,0,a(-r^*,0),b(-r^*,0)).
\]
Since $\sigma_2(\pi)=0$, we have $\sigma(\pi) \in S$. On the other hand, as $\sigma(\pi)\in W^u(\mathcal{O})$, we obtain $H(\sigma(\pi))=0$ and, consequently, $\sigma(\pi)\in C$. From the expressions given in \eqref{eq:plano_tangente} for the tangent plane to $W^u(\mathcal{O})$ at $\mathcal{O}$, it follows that $a(-r^*,0)=(\rho^2+\omega^2)r^*+O((r^*)^2)$. Therefore, $a(-r^*,0)>0$ if $r^*$ is small enough. Thus, as $\sigma_1(\pi)<0$ and $\sigma_3(\pi)>0$, we conclude that $\sigma(\pi) \in C_1$. Arguments to prove that $\sigma(2\pi) \in C_3$ are analogous.

(2) We write a first order expansion of $x_2(t)$,
\[
x_2(t)=x_2(0)+x'_2(0)\, t+O(t^2)=x_2(0)+x_3(0)\, t+O(t^2).
\]
If we consider the solution with initial condition $\sigma(\pi)=(-r^*,0,a(-r^*,0),b(-r^*,0))$, then
\[
x_2(t)=a(-r^*,0)\, t+O(t^2).
\]
As argued above, $a(-r^*,0)>0$ if $r^*$ is small enough. Therefore, if $t$ is small enough, $x_2(t)>0$. For the initial condition $\sigma(2\pi)=(r^*,0,a(r^*,0),b(r^*,0))$, we have $a(r^*,0)<0$ for $r^*$ small, and hence the result follows the same lines.

(3) Let $\varphi(t,\bar x)=(x_1(t),x_2(t),x_3(t),x_4(t))$. If $x_2(t)>0$ for all $t \in [0,\tau(\bar x)[$ and $x_2(\tau(\bar x))=0$, then $x_2$ cannot be strictly increasing at $\tau(\bar x)$, i.e. $x_3(\tau(\bar x))=x'_2(\tau(\bar x))>0$ cannot occur and hence $\bar x^* = \varphi(\tau(\bar x),\bar x) \notin C_1 \cup C_2$. The case with $x_2(t)<0$ for all $t \in [0,\tau(\bar x)[$ and $x_2(\tau(\bar x))=0$ is argued in a similar manner to prove that $\bar x^* \notin C_3 \cup C_4$.

(4) Assume that $\bar x \in (C_1 \cup C_2) \cap \Sigma_k^{u}$ and let $\varphi(t,\bar x)=(x_1(t),x_2(t),x_3(t),x_4(t))$. A first order expansion of $x_2(t)$ is given by
\[
x_2(t)=x_2(0)+x'_2(0)\, t+O(t^2)=x_3(0)\, t+O(t^2).
\]
Since $\bar x \in C_1 \cup C_2$, by definition we have $x_3(0)>0$ and therefore $x_2(t)>0$ for $t>0$ small enough. Now, we apply Property 3 to get the result. The case $\bar x \in (C_3 \cup C_4) \cap \Sigma_k^{u}$ is proved in an analogous manner.
\end{proof}

Let us denote $q_0^+=\sigma(\pi) \in D^{u,+}$ and $q_0^-=\sigma(2\pi) \in D^{u,-}$. Moreover, we write $q_{k}^\pm=\varphi(t_{k}(q_{k-1}^\pm),q_{k-1}^\pm)\in \Sigma_{k}^{u,\pm}$ for each $k\geq 1$. In Figure \ref{fig:cinco_intersecciones}, we show $q_k^+$, for $k=0,1,2$, and $q_k^-$, for $k=0,\ldots,5$, when $\eta_3=-1.73$. According to (1) in Proposition \ref{prop:prop_basicas}, $q_0^+ \in C_1$ and $q_0^- \in C_3$ (the two black dots in Figure \ref{fig:cinco_intersecciones}). Notice that they do not belong to $\Sigma^u$. From properties (2) and (3) in Proposition \ref{prop:prop_basicas}, it follows that $q_1^+  \in C_3 \cup C_4$ and $q_1^-  \in C_1 \cup C_2$. In fact, we observe that $q_1^+  \in C_3$ and $q_1^- \in C_1$. Subsequent alternance between $C_1 \cup C_2$ and $C_3 \cup C_4$ is a direct consequence of the property (4) in Proposition \ref{prop:prop_basicas}. In particular, we get that $q_2^+ \in C_3$, $q_2^-,q_4^- \in C_3$, $q_3^-\in C_1$ and $q_5^- \in C_2$. We also point out that the positive orbit of $q_2^+$ has no more intersections with $S$. According to the numerical results, it becomes an unbounded orbit. In the same way, the sets $\Sigma^{u,\pm}_k$ also illustrate the rule provided by property (4) in Proposition \ref{prop:prop_basicas}. 

\begin{figure}
    \centering
    \includegraphics[width=0.7\linewidth]{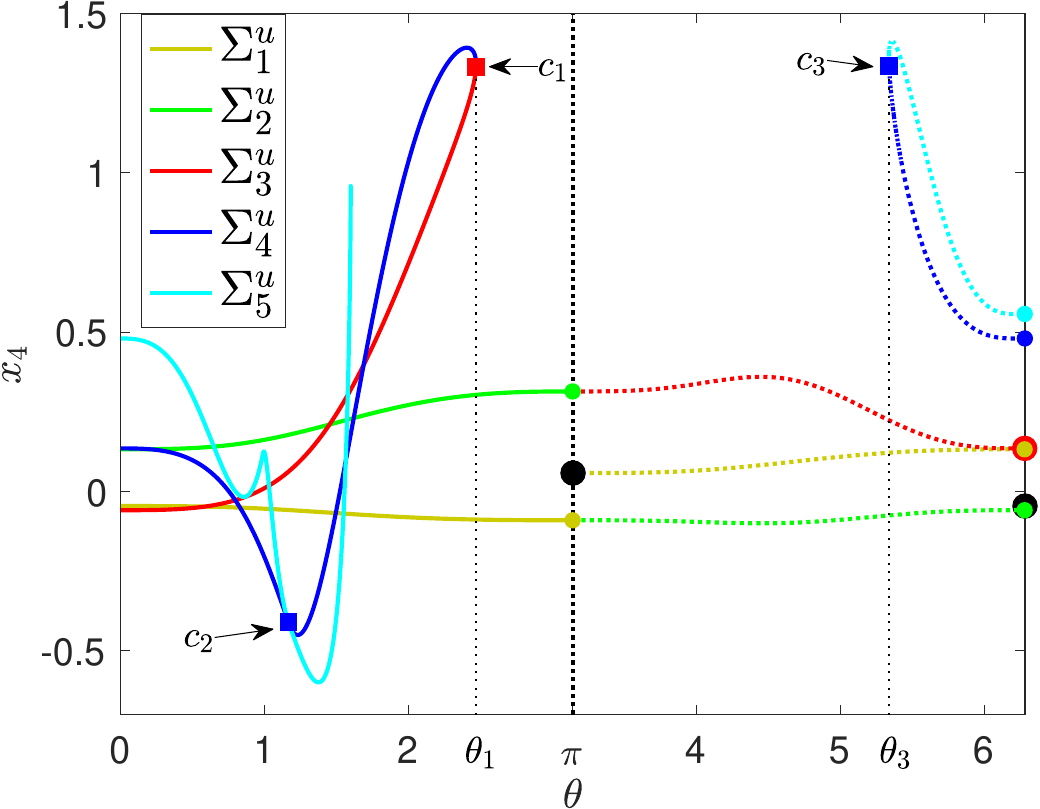}
    \caption{A plot of the values of $x_4$ along $\Sigma^u_i$, with $i=1,\ldots,5$, as a function of the angle that parameterizes $D^u$. Solid (respectively dotted) lines correspond to the intersections of $D^{u,+}$ (respectively $D^{u,-}$). This figure helps to a better understanding of Figure \ref{fig:cinco_intersecciones}.}
    \label{fig:angulos}
\end{figure}
For a full understanding of Figure \ref{fig:cinco_intersecciones}, we provide the Figure \ref{fig:angulos}. For each $\theta \in ]0,2\pi]$, we plot $\Psi_4(\varphi(t_k(\sigma(\theta)),\sigma(\theta)))$ for different values of $k$, where $\Psi_4$ denotes the projection on the $x_4$-axis. The piece $\Sigma_1^{u,+} \subset C_3$ (solid olive green line) is mapped into $\Sigma_2^{u,+} \subset C_1$ (solid green line) and a part of this one, at least, is mapped into $\Sigma_3^{u,+} \subset C_3$ (solid red line). There is no third intersection for angles in $]\theta_1,\pi]$ (see Figure \ref{fig:angulos}). Next, $\Sigma_4^{u,+}$ (solid blue line) is a continuous line on $C$, but it splits into two pieces, one of them contained in $C_1$ and the other in $C_2$. Finally, $\Sigma_5^{u,+}$ (solid cyan line) is a continuous line on $C$ that splits into three pieces, one of them contained in $C_4$ and the others in $C_3$. On the other hand, we have that the piece $\Sigma_1^{u,-} \subset C_1$ (dotted olive green line) is mapped into $\Sigma_2^{u,-} \subset C_3$ (dotted green line), and this one into $\Sigma_3^{u,-} \subset C_1$ (dotted red line). And a part of $\Sigma_3^{u,-}$ is mapped into $\Sigma_4^{u,-} \subset C_3$ (dotted blue line). Note that there is no fourth intersection for angles in $]\pi,\theta_3]$ (see Figure \ref{fig:angulos}). Finally, we also show that $\Sigma_5^{u,-} \subset C_2$ (solid cyan line).

Results below are useful to understand the geometry of $\Sigma^u$ at the points $c_1$, $c_2$ and $c_3$ (see Figures \ref{fig:cinco_intersecciones} and \ref{fig:angulos}). In the sequel, $\Pi$ stands for the Poincaré map from $C$ to $C$, whenever it is well defined. 

\begin{proposition}
\label{prop:corte_v1_arriba}
    Assume that there exists a point $q=(0,0,0,q_4)\in\Sigma^u_k \cap v_1$, with $q_4>0$, where $W^u(\mathcal{O})$ intersects transversely $v_1$. Then, locally around $q$, $\Sigma^u_k\setminus \{q\} = \Sigma^{u,3}_k \cup \Sigma^{u,4}_k$, with $\Sigma^{u,m}_k \subset C_m$ for $m=3,4$. Moreover, there exist two open arcs $\Sigma^{u,m}_{k+1} \subset \Sigma^u_{k+1} \cap C_m$, with $m=1,2$, such that $\Sigma^{u,1}_{k+1}=\Pi(\Sigma^{u,4}_{k})$, $\Sigma^{u,2}_{k+1}=\Pi(\Sigma^{u,3}_{k})$ and $q$ is a common end point of both $\Sigma^{u,1}_{k+1}$ and $\Sigma^{u,2}_{k+1}$.
\end{proposition}
\begin{proof}
    In a neighbourhood of $q$, the equation $H(x_1,x_2,x_3,x_4)=0$ defines $x_2$ as a function:
    \[
    x_2=f(x_1,x_3,x_4)=
    \left\{\begin{aligned}
    \dfrac{x_4-\sqrt{x_4^2+2\eta_3
    \left(\frac{1}{2}x_1^2-\frac{1}{3}x_1^3-\frac{1}{2}x_3^2\right)}}{\eta_3}
    & \quad \mbox{if} \quad \eta_3 \neq 0,
    \\[1ex]
    \dfrac{-\left(\frac{1}{2}x_1^2-\frac{1}{3}x_1^3-\frac{1}{2}x_3^2\right)}{x_4}
    & \quad \mbox{if} \quad \eta_3 = 0.
    \end{aligned}\right.
    \]
    Hence, the equations of the vector field reduced to the level zero set around $q$ can be written as follows:
    \begin{equation*}
    \left\{\begin{array}{l}
    x_1' = f(x_1,x_3,x_4),\\
    x_3' = x_4,\\
    x_4' = - x_1 + \eta_3 x_3 + x_1^2.
    \end{array}\right.
    \end{equation*}
    In this three-dimensional phase space, the cross section $S$ reduces to the surface of the loop cylinder $C$ in a neighborhood of $q$. Since $W^u(\mathcal{O})$ is transversal to the vertical line $v_1$ at $(x_1,x_3,x_4)=(0,0,q_4)\equiv\hat q$, we have that $(W^u(\mathcal{O})\setminus \{\hat q\}) \cap C$ splits into four curves $\gamma_m \subset C_m$, with $m=1,\dots,4$. Since the reduced vector field at $\hat q$ is $(0,q_4,0)$ and $q_4>0$, $x_3$ increases locally. Therefore, the forward flow sends $\gamma_4$ to $\gamma_1$ and $\gamma_3$ to $\gamma_2$. It follows that $\Sigma_k^{u,m}=\gamma_m$ for $m=3,4$, and $\Sigma_{k+1}^{u,m}=\gamma_m$ for $m=1,2$.
\end{proof}
We can prove, in a similar way, that the proposition below is also true.
\begin{proposition}
\label{prop:corte_v1_abajo}
    Assume that there exists a point $q=(0,0,0,q_4)\in\Sigma^u_k \cap v_1$, with $q_4<0$, where $W^u(\mathcal{O})$ intersects transversely $v_1$. Then, locally around $q$, $\Sigma^u_k\setminus \{q\} = \Sigma^{u,1}_k \cup \Sigma^{u,2}_k$ with $\Sigma^{u,m}_k \subset C_m$ for $m=1,2$. Moreover, there exist two open arcs $\Sigma^{u,m}_{k+1} \subset \Sigma^u_{k+1}$, with $m=3,4$, such that $\Sigma^{u,4}_{k+1}=\Pi(\Sigma^{u,1}_{k})$, $\Sigma^{u,3}_{k+1}=\Pi(\Sigma^{u,2}_{k})$ and $q$ is a common end point of both $\Sigma^{u,3}_{k+1}$ and $\Sigma^{u,4}_{k+1}$.
\end{proposition}

\begin{proposition}
\label{prop:corte_v2_arriba}
    Assume that there exists a point $q=(3/2,0,0,q_4)\in\Sigma^u_k \cap v_2$, with $q_4>0$, where $W^u(\mathcal{O})$ intersects transversely $v_2$. Then, locally around $q$, $\Sigma^u_k\setminus \{q\}$ consists of an open arc $\Sigma^{u,3}_k \subset C_3$. Moreover, there exists an open arc $\Sigma^{u,2}_{k+1} \subset \Sigma^u_{k+1}$, having $q$ as an end point, such that $\Sigma^{u,2}_{k+1}=\Pi(\Sigma^{u,3}_{k})$ and $q$ is a common end point of both $\Sigma^{u,2}_{k+1}$ and $\Sigma^{u,3}_k$.
\end{proposition}
\begin{proof}
As in the proof of Proposition \ref{prop:corte_v1_arriba}, we consider a reduction to a three-dimensional space by writing $x_2$ as a function of $(x_1,x_3,x_4)$ on the zero level set of $H$, in a neighbourhood of $q$. In the three-dimensional phase space, the cross section $S$ reduces to the surface of the loop cylinder $C$ in a neighborhood of $q$. Since $W^u(\mathcal{O})$ is transversal to the vertical line $v_2$ at $(x_1,x_3,x_4)=(3/2,0,q_4)\equiv\hat q$, we have that $(W^u(\mathcal{O})\setminus \{\hat q\}) \cap C$ splits into two curves $\gamma_m \subset C_m$, with $m=2,3$. Since the reduced vector field at $\hat q$ is $(0,q_4,0)$ and $q_4>0$, $x_3$ increases locally. Therefore, the forward flow sends $\gamma_3$ to $\gamma_2$. It follows that $\Sigma_k^{u,2}=\gamma_2$ and $\Sigma_{k+1}^{u,3}=\gamma_3$.
\end{proof}
We can prove, in a similar way, that the preposition below is also true.
\begin{proposition}
\label{prop:corte_v2_abajo}
    Assume that there exists a point $q=(3/2,0,0,q_4)\in\Sigma^u_k \cap v_2$, with $q_4<0$, where $W^u(\mathcal{O})$ intersects transversely $v_2$. Then, locally around $q$, $\Sigma^u_k\setminus \{q\}$ consists of an open arc $\Sigma^{u,2}_k \subset C_2$. Moreover, there exists an open arc $\Sigma^{u,3}_{k+1} \subset \Sigma^u_{k+1}$, having $q$ as an end point, such that $\Sigma^{u,3}_{k+1}=\Pi(\Sigma^{u,2}_{k})$ and $q$ is a common end point of both $\Sigma^{u,3}_{k+1}$ and $\Sigma^{u,2}_k$.
\end{proposition}

In Figure \ref{fig:cinco_intersecciones}, at the point $c_2$ on $\Sigma_4^u \cap v_1$ (blue square at $\bar x_1=-1.5$), Proposition \ref{prop:corte_v1_abajo} can be applied, where $c_2$ takes the role of $q$ in the statement and $q_4 < 0$. Locally around $c_2$, two pieces on $\Sigma_4^u$ are distinguished: $\Sigma_4^{u,1} \subset C_1$ and $\Sigma_4^{u,2} \subset C_2$. The former is mapped into a curve $\Sigma_5^{u,4} \subset C_4$, whereas $\Sigma_4^{u,2}$ is mapped into a curve $\Sigma_5^{u,3} \subset C_3$. Moreover, we can observe how the counterpart of $c_2$ at $\bar x_1=1.5$ indeed serves as a common endpoint for $\Sigma_5^{u,3}$ and $\Sigma_5^{u,4}$. It is important to notice that this counterpart of $c_2$ does not belong to $\Sigma_5^u$. Hence, Proposition \ref{prop:corte_v1_abajo} does not apply to $\Sigma_5^u$ in a neighborhood of $c_2$. Nevertheless, $\Sigma_5^u$ intersects $v_1$ at a point $c_4$ (light blue square in the figure), where Proposition \ref{prop:corte_v1_arriba} is applicable. Authors in \cite{lau1992} and \cite{bufchatol1996} considered points similar to $c_2$ as double intersection points, which implies that $c_2$ is regarded, for them, as both the fourth and fifth cross.

On the other hand, at point $c_1$ (red square at $\bar x_1=0$), Proposition \ref{prop:corte_v2_arriba} can be applied. The piece of $\Sigma_3^{u}$ contained in $C_3$ with an end point at $c_1$ is mapped into the piece of $\Sigma_4^{u}$ contained in $C_2$ with an end point at $c_1$. This point $c_1$ is a turning point. The geometry around $c_3$ (big blue square behind $c_1$) is similar but involving $\Sigma_4^u$ and $\Sigma_5^u$.

\subsection{Homoclinic orbits}
Given $m$, $n\in \mathbb{N}$, each intersection point $q \in \Sigma^u_m \cap \Sigma^s_n$, if it exists, corresponds to a homoclinic orbit of the system. Moreover, the intersection between $\Sigma^u_m$ and $\Sigma^s_n$ determines the intersection between $\Sigma^u_{m+i}$ and $\Sigma^s_{n-i}$, for all $i=-m+1,-m +2,\ldots,n-2,n-1$. 
\begin{proposition}
 Let $k,l,m,n \in \mathbb{N}$ and assume that $k+l=m+n$. Then $\# (\Sigma^u_k \cap \Sigma^s_l)=\# (\Sigma^u_m \cap \Sigma^s_n)$.
\end{proposition}
\begin{proof}
Assume that $k<m$ and let $q \in \Sigma^u_k \cap \Sigma^s_l$. Then $\Pi^j(q)\in \Sigma^u_{k+j} \cap \Sigma^s_{l-j}$ for $j=1,\ldots,m-k$ . In particular, we get $\Pi^{m-k}(q)\in \Sigma^u_{m} \cap \Sigma^s_{n}$ when $j=m-k$. Conversely, if $q \in \Sigma^u_m \cap \Sigma^s_n$. Then $\Pi^{-j}(q)\in \Sigma^u_{m-j} \cap \Sigma^s_{n+j}$ for $j=1,\ldots,m-k$. In particular, we obtain $\Pi^{m-k}(q)\in \Sigma^u_{k} \cap \Sigma^s_{l}$ when $j=m-k$. The case $k>m$ is analogous.
\end{proof}

Each homoclinic orbit is identified with the corresponding intersections between $\Sigma^u$ and $\Sigma^s$. We say that a homoclinic orbit $\Gamma$ has \emph{order} $\alpha \in \mathbb{N}$ if $\Gamma \cap \Sigma^u_1 \cap \Sigma^s_{\alpha-1} \neq \varnothing$. In this case, there are (unique) intersections of $\Gamma$ with $\Sigma^u_i \cap \Sigma^s_{\alpha-i}$ for all $i=1,\ldots,\alpha-1$. For simplicity, we identify the homoclinic orbit of order $\alpha$ by
\[
\Gamma_\alpha\equiv [q_{i,\alpha-i}]_{i=1,\ldots,\alpha-1} = [q_{1,\alpha-1},q_{2,\alpha-2},\ldots,q_{\alpha-1,1}] ,
\]
where $q_{i,\alpha-i}$ is the (only) cross point of $\Gamma_\alpha$ with $\Sigma^u_i \cap \Sigma^s_{\alpha-i}$, for each $i =1,\ldots,\alpha-1$. The subscript will indicate the order of a homoclinic orbit whenever appropriate.

Moreover, we distinguish between symmetric and asymmetric homoclinic orbits.
\begin{definition}
     We say that a homoclinic orbit $\Gamma$ is symmetric if $R(\Gamma)=\Gamma$. Otherwise, we will say that $\Gamma$ is asymmetric.
\end{definition}

According to \cite[Lemma 3]{vanfie1992}, an orbit of a time-reversible system is symmetric if and only if it intersects $\mathrm{Fix}(R)$. As a consequence, each point $q \in \Sigma^u \cap \mathrm{Fix}(R)$ corresponds to a symmetric homoclinic orbit. Indeed, by symmetry $q= R(q) \in R(\Sigma^u)=\Sigma^s$ and therefore the orbit of $q$ is homoclinic and symmetric because it intersects $\mathrm{Fix}(R)$. Furthermore, as argued in \cite{har1998}, the intersection point of each symmetric homoclinic orbit with $\mathrm{Fix}(R)$ is unique. Thus, there is a one-to-one correspondence between symmetric homoclinic orbits and intersections between $\Sigma^u$ and $\mathrm{Fix}(R)$.

Assume that a homoclinic orbit of order $\alpha$, $\Gamma_\alpha$,
is symmetric. Therefore, there is a unique $q_{i,\alpha-i}\in \mathrm{Fix}(R)$. Since $q_{i,\alpha-i}\in \Sigma^u_i \cap \mathrm{Fix}(R)$, it follows that $q_{i,\alpha-i} = R(q_{i,\alpha-i}) \in R(\Sigma^u_i)=\Sigma^s_i$. On the other hand, by definition, $q_{i,\alpha-i} \in \Sigma^s_{\alpha-i}$ and taking into account that $\Sigma^s_i \cap \Sigma^s_j = \varnothing$ if $i\neq j$, we conclude that $i=\alpha-i$, that is, $\alpha$ is even. Moreover, with the same argument, given $\alpha$ even, if $q \in \Sigma^u_{\alpha/2} \cap \Sigma^s_{\alpha/2}$, but $q \notin \mathrm{Fix}(R)$, the homoclinic orbit through $q$ is asymmetric.

\begin{figure}[t]
    \centering
    \includegraphics[width=0.85\linewidth]{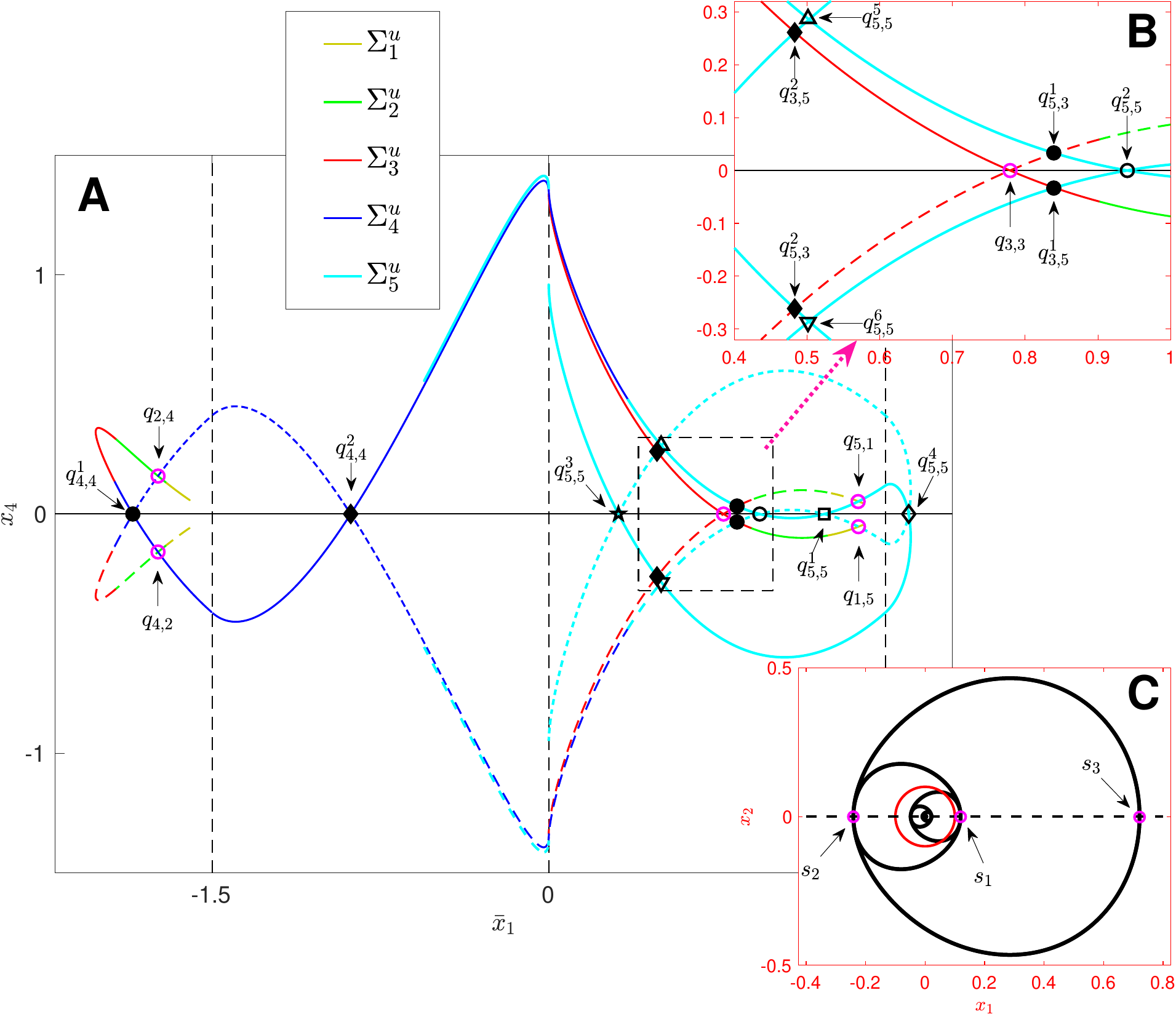}
    \caption{The first five intersections $\Sigma^{u}_k$ (solid lines) and  $\Sigma^{s}_k$ (dotted lines), for $k=1,\ldots,5$ and $\eta_3=-1.73$ on $P(C)$. We use the notation $q_{i,j}$ to indicate that a point belongs to $\Sigma^u_i \cap \Sigma^u_j$. Superindices are used when $\Sigma^u_i \cap \Sigma^u_j$ contains more than one point to differentiate between homoclinic orbits. See the text for additional explanations. Panel B is and enlargement of the rectangle show in Panel A. Panel C shows the projection of the homoclinic orbit of order $6$ (see non-filled magenta circles) on the plane $(x_1,x_2)$.}
    \label{fig:cincointersecciones}
\end{figure}
Figure \ref{fig:cincointersecciones} illustrates the homoclinic orbits detected from the calculation of $\Sigma^u_{i}$ with $i=1,\ldots,5$, when $\eta_3= -1.73$. Solid lines represent intersections $\Sigma^u_{i}$, and dashed lines correspond to symmetric $\Sigma^s_{i}$. With the information provided by these intersections, seven symmetric homoclinic orbits are observed. Panels A and B show a single symmetric homoclinic orbit (nonfilled magenta circles) of order $6$, $\Gamma_6 \equiv [q_{1,5},q_{2,4},q_{3,3},q_{4,2},q_{5,1}]$.
This is the only homoclinic orbit for which all of its intersections with $\Sigma^u$ and $\Sigma^s$ are displayed. Panel C shows the homoclinic orbit (black) projected onto the coordinate plane $(x_1,x_2)$ and, in red, the projection of the fundamental domains. Note that the first intersection after leaving the fundamental domain $D^u$ is the point $q_{1,5} \in \Sigma^u_1$ that comes from a point of $\sigma(\theta)\in D ^{u,+}$. In the panel C, $\sigma(\theta)$ corresponds to the intersection with $x_2>0$, between the projected orbit and the red circle. The projected curve then successively passes through $s_1$, $s_2$, $s_3$, $s_2$ and $s_1$ before crossing the red circle again, the instant at which the homoclinic orbit enters $D^s$. We also notice that $s_1$ is the projection of $q_{1,5}$ and $q_{5,1}$, $s_2$ the projection of $q_{2,4}$ and $q_{4,2}$, and $s_3$ the projection of $q_{3,3}$.

There are two symmetric homoclinic orbits of order $8$:
\[\Gamma^\nu_8 \equiv [q^\nu_{i,8-i}]_{i=1,\ldots,7},\]
with $\nu=1,2$.
Points $q^\nu_{3,5},q^\nu_{4,4},q^\nu_{5,3}$, with $\nu=1,2$, are shown in panels A and B (black circles for $\nu=1$ and black diamonds for $\nu=2$).

Moreover, there are four symmetric homoclinic orbits of order $10$:
\[\Gamma^\nu_{10} \equiv [q^\nu_{i,10-i}]_{i=1,\ldots,9},\]
with $\nu=1,2,3,4$.
The points $q^\nu_{5,5}$, with $\nu=1,2,3,4$, are shown in panels A and B with nonfilled black square ($\nu=1$), circle ($\nu=2$), star ($\nu=3$) and diamond ($\nu=4$).

Finally, there are two asymmetric homoclinic orbits of order $10$:
\[\Gamma^\nu_{10} \equiv [q^\nu_{i,10-i}]_{i=1,\ldots,9},\]
with $\nu=5,6$.
The points $q^\nu_{5,5}$, with $\nu=5,6$, are shown in panels A and B with a black non-filled and non-inverted triangle ($\nu=5$) and a black non-filled inverted triangle ($\nu=6$).

\begin{remark}
     Our labeling of homoclinic orbits differs from that proposed in \cite{bufchatol1996}. 
Roughly speaking, and taking into account the behavior of the graph of $x_1(t)$ along the orbit, whenever it makes an excursion outside a certain fixed neighborhood of the origin, $x_1$ achieves a high amplitude maximum. Labelling in \cite{bufchatol1996} focuses on the number of these maxima, as well as the count of intersections with $x_1=0$ and the number of local extrema of $x_1$ between each pair of consecutive large maxima. 
\end{remark}

\section{A discussion about intersections}
\label{sec:intersections}
After introducing our approach to visualize the intersections of the invariant manifolds at the origin with the crossing section ${x_2=0}$ and exploring its implications in the discussion of the existence and genesis of homoclinic bifocal orbits, this section provides a concise overview of how the initial intersections evolve with respect to the parameter.

\begin{figure}
    \centering
    \includegraphics[width=0.97\linewidth]{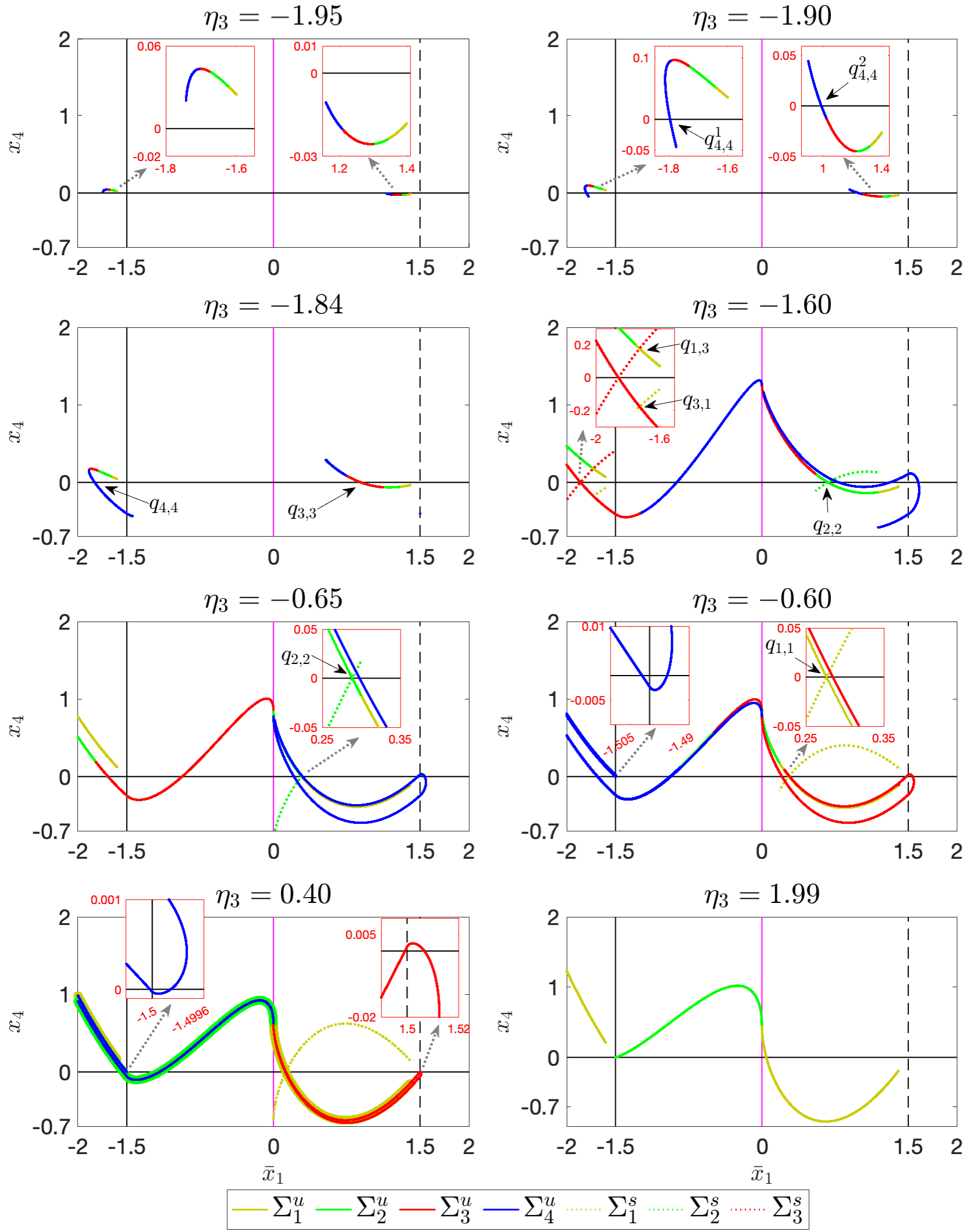}
    \caption{Intersections $\Sigma_i^u$, for $i=1,2,3,4$, are depicted for several parameter values. Two prominent behaviors emerge: first, a tendency for the initial intersections to overlap as $\eta_3$ approaches $2$; second, the reduction in the number of intersections exhibited by the primary homoclinic as it transits from $D^u$ to $D^s$. Additional details are provided in the text.}
    \label{fig:cuatro_cortes}
\end{figure}

Figure \ref{fig:cuatro_cortes} illustrates the intersections $\Sigma_i^u$, with $i=1,2,3,4$ (solid lines), as $\eta_3$ increases from $\eta_3=-1.95$, a value close to $\eta_3=-2$, where the eigenvalues at equilibrium become pure resonant imaginaries, to $\eta_3=1.99$, a value close to $\eta_3=2$, where the eigenvalues become real. In particular, one change observed with growing $\eta_3$ is that, although it is typical for successive intersections to overlap as we approach $\eta_3=2$, the first intersections themselves appear coincident when they are plotted together, as evidenced in the intersections for $\eta_3=0.40$, where we have used different thicknesses to be able to see all the curves. It is also worth mentioning that obtaining the intersections $\Sigma_i^u$ for large $i$ is numerically challenging. This difficulty arises because the length of the arcs in the fundamental domain that determine each intersection tends to $0$ as $i$ increases. For values of $\eta_3$ near $2$, these lengths escape from reasonable discretization ranges for the angle in the fundamental domain, even for small $i$. In the case of $\eta_3=1.99$, only intersections $\Sigma_1^u$ and $\Sigma_2^u$ were obtained. For $\eta_3=0.40$, we observe the first four intersections, but with overlaps. As a consequence, the tentacular geometry becomes difficult to visualize for values of $\eta_3$ close to $2$.

Another interesting aspect illustrated in Figure \ref{fig:cuatro_cortes}, is the evolution of the primary homoclinic orbit, the one with the lowest-order intersection. For $\eta_3=-1.95$, the intersection with $\mathrm{Fix} (R)$ occurs for $\Sigma_i^u$ with $i>4$ (not plotted). It seems that the number of intersections required to achieve a first cross increases monotonically and without bounds as $\eta_3$ approaches $-2$. Simultaneously, the length of the intersections curves decreases. An issue of great interest is to analyze the limit set of $\cup_{i\in\mathbb{N}} \Sigma_i^u$ as $\eta_3$ approaches $-2$, the point where the hyperbolicity of the equilibrium is lost as the eigenvalues collapse onto the imaginary axis.

For $\eta_3=-1.90$, we observe two crosses between $\Sigma_4^u$ and $\mathrm{Fix}(R)$, that is, two homoclinic orbits of order $8$ corresponding to intersection points $q_{4, 4}^\nu$, with $\nu=1,2$, marked in the figure. The first one at $q_{4, 4}^1$ still persists for $\eta_3=-1.84$, but the other one has changed to order $6$ (point $q_{3,3}$ in the figure), and this becomes a homoclinic orbit of order $4$ when $\eta_3=-1.60$. In this case, all the intersections of this homoclinic orbit with $\cup_{i\in\mathbb{N}} \Sigma_i^u$ are shown in the figure: $[q_{1,3},q_{2,2},q_{3,1}]$. For $\eta_3=-0.60$, we identify a homoclinic orbit of order $2$ (point $q_{1,1}$) that persists for all parameter values as $\eta_3$ increases and tends to $2$.

Figure \ref{fig:cuatro_cortes} also illustrates the emergence of tentacles, which are consistently present for all $\eta_3 \in ]-2,2[$ due to their inherent connection to bifocal homoclinic orbits. However, as previously discussed, obtaining the crossings $\Sigma_i^u$ becomes notably intricate as the parameter values approach $2$, even for small $i$ values. Defining a tentacle accurately is challenging, but broadly, it can be characterized as any $T \subset \Sigma_i^u$ bounded by two intersections with $\mathrm{Fix}(R)$. For $\eta_3=0.40$, two tentacles are depicted: one in $\Sigma^u_3$ (in red) and the other in $\Sigma^i_4$ (in blue). These tentacles persist for $\eta_3=-0.60$, but when $\eta_3=-0.65$ is reached, the blue tentacle on the left side vanishes and the red tentacle changes to blue.

\section{Cascades of homoclinic tangencies}
\label{sec:cocoon}

Figure \ref{fig:tentacles} shows a large number of foldings of $\Sigma^u$ and $\Sigma^s$. As we have already mentioned, following the terminology introduced in \cite{bufchatol1996}, we refer to those folds as tentacles. When we vary the parameter, the tentacles wrap and unwrap, giving rise to tangencies with $\mathrm{Fix}(R)$. In the unfolding of a homoclinic tangency with respect to the parameter $\eta_3$, two homoclinic orbits collide and disappear. Figure \ref{fig:homoclinic_tangency} illustrates one of these unfoldings for $\eta_3 \approx -1.776$. It shows three tentacles formed by $\Sigma^u_{10}$ for three different values of the parameter. For $\eta_3 \approx -1.771$ two cross points $q^1_{10,10}$ and $q^2_{10,10}$ are observed between $\Sigma^u_{10}$ and $\mathrm {Fix}(R)$, which corresponds to two homoclinic orbits of order $20$. As $\eta_3$ decreases, these points get closer and closer until they collapse to a single point for $\eta_3 \approx -1.776$ (green curve). For $\eta_3\approx -1.781$ the part of $\Sigma^u_{10}$ shown (red curve) no longer exhibits cross points.
These homoclinic tangencies correspond to the coalescence of symmetric homoclinic orbits described in \cite[Section 4.2 (Case 1) and Figure 15]{bufchatol1996}.
\begin{figure}[t]
    \centering
    \includegraphics[width=0.75\linewidth]{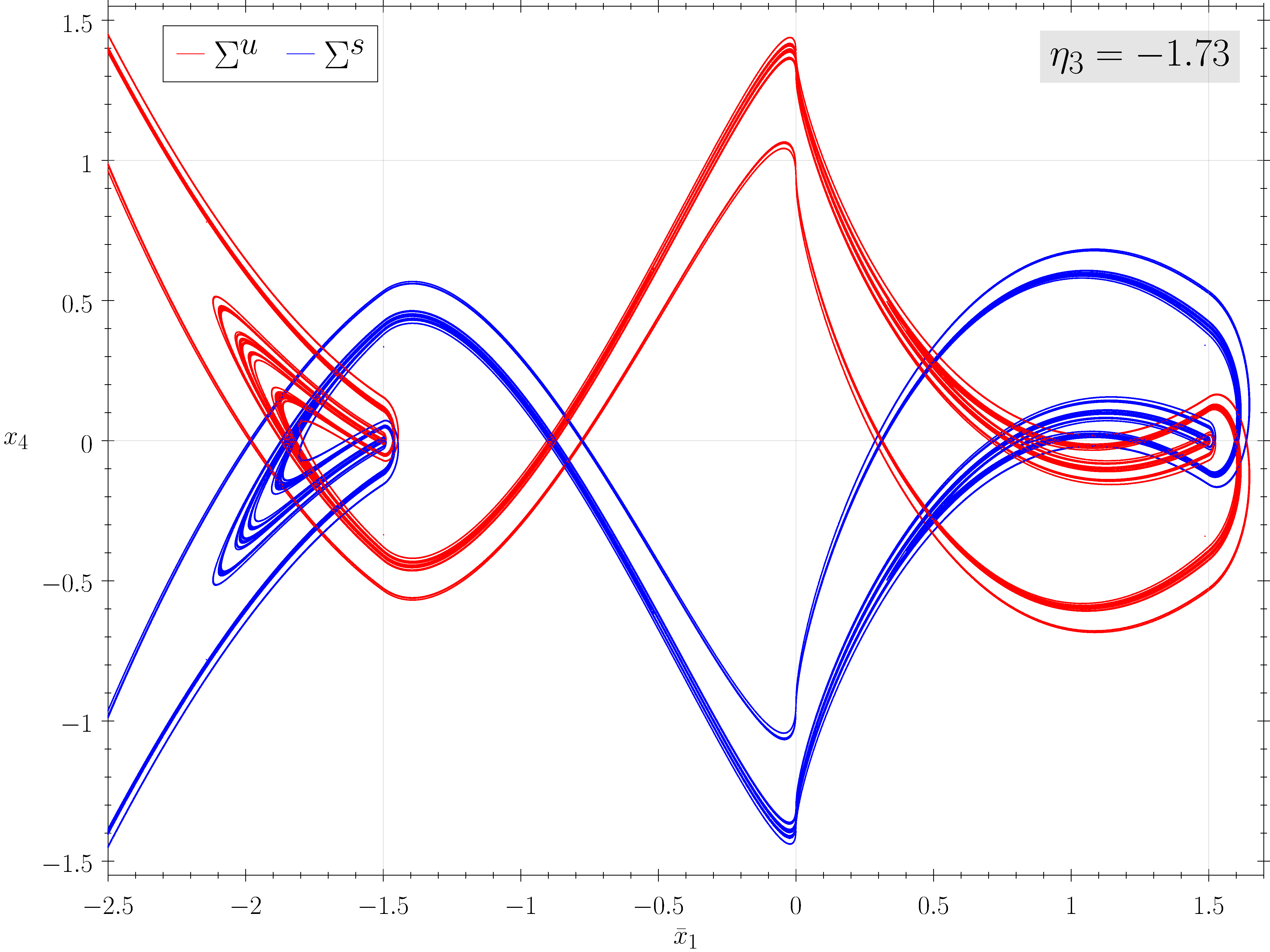}
    \caption{Sets $\Sigma^u_k$ and $\Sigma^s_k$ for $\eta_3=-1.73$ (a large number $k$ of intersections is considered). This plot provides a clear illustration of the existence of infinitely many tentacles and, as is already known theoretically, infinitely many homoclinic orbits.}
    \label{fig:tentacles}
\end{figure}

The geometry of the local transitions around the equilibrium and the numerical evidence itself suggest that, as tentacles wrap and unwrap over the loop cylinder $C$, an infinite number of homoclinic tangencies similar to the one we have just discussed will emerge. This is now the point where we link the present problem with the study of cocoon bifurcations conducted in \cite{dumibakok2006,lau1992}.
\begin{figure}
    \centering
    \includegraphics[width=0.75\linewidth]{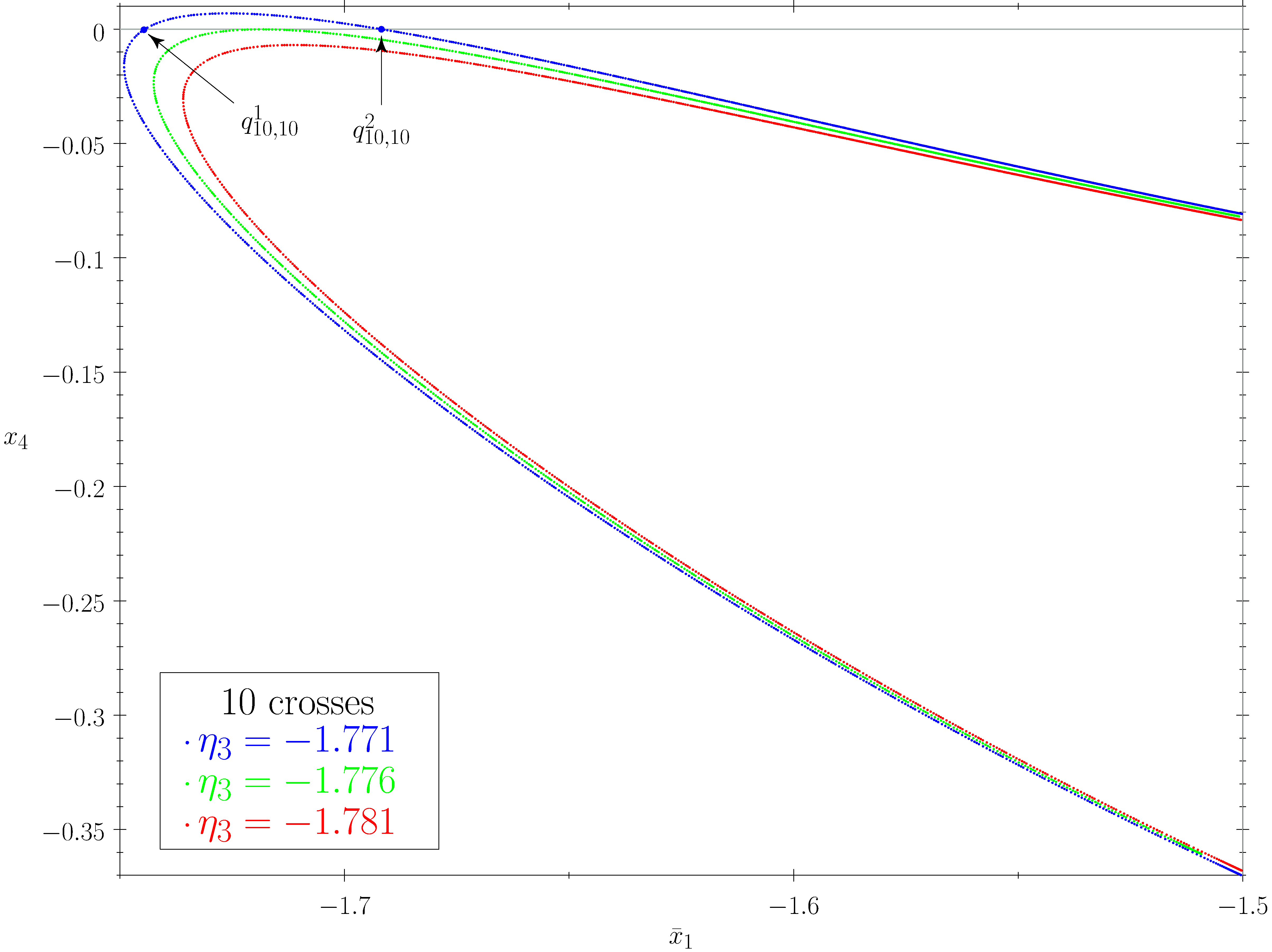}
    \caption{Pieces of $\Sigma_{10}^u$ for three different values of the parameter. When $\eta_3\approx -1.771$, $\Sigma_{10}^u$ has two intersections with $\mathrm{Fix}(R)$. These homoclinic orbits collide and disappear at $\eta_3\approx -1.776$, where there is a tangent bifurcation.}
    \label{fig:homoclinic_tangency}
\end{figure}

\subsection{Cocooning cascades of homoclinic tangencies}
The set of phenomena that give shape to what Lau called the cocoon bifurcation was explained in \cite{lau1992}. In \cite{dumibakok2006}, an organizing center was introduced for part of the phenomena related to cocoon bifurcation. The cocoon bifurcation takes place in the Michelson system \eqref{michelson_system}, which is a three-dimensional and reversible system. However, the reversibility of the four-dimensional system \eqref{sistema_principal}, together with the fact that we work on the three-dimensional manifold $H^{-1}(0)$, allow us to place ourselves in a scenario parallel to that of Michelson system. Indeed, it is possible to develop the theoretical framework in a general context, following similar ideas to those in \cite{dumibakok2006}.

Let $X_\lambda$ be a one-parameter family of vector fields in $\mathbb{R}^4$, verifying the properties below:
\begin{description}
    \item[\textrm{(H1)}] $X_\lambda$ is time-reversible with respect to the linear involution $R$ with $\mathrm{dim}(\mathrm{Fix}(R))=2$, where $\mathrm{Fix}(R)$ stands for the  subspace of fixed points of $R$.
    \item[\textrm{(H2)}] There exists a first integral $H_\lambda$ for the vector field $X_\lambda$.
    \item[\textrm{(H3)}] For all $\lambda$, $X_\lambda$ has a hyperbolic equilibrium point $E\in \mathrm{Fix}(R)$ such that $\mathrm{dim} \, W^u(E)=\mathrm{dim} \, W^s(E)=2$.
\end{description}
Without loss of generality, we assume that the involution $R$ is given by the map \eqref{reversor}, and also that $H(E)=0$. Moreover, we impose the following transversality condition:
\begin{description}
    \item[\textrm{(H4)}] $(H^{-1}(0)\setminus \{E\}) \cap \mathrm{Fix}(R) \neq \varnothing$,
    and $\frac{\partial H}{\partial x_1}(q) \frac{\partial H}{\partial x_3}(q) \neq 0$ for all $q\in H^{-1}(0) \cap \mathrm{Fix}(R)$, with $q\neq E$.
\end{description}
\begin{remark}
    Condition (H4) means that, out of $E$, the level set $H^{-1}(0)$ and $\mathrm{Fix}(R)$ meet transversely.
\end{remark}
\begin{definition}
    Under the conditions (H1)-(H3), we say that the family $X_\lambda$ has a \textbf{cocooning cascade of homoclinic tangencies} centred at $\lambda_*$ if there is a closed solid 2-torus $T \subset H^{-1}(0)$ with $E \notin T$ and a monotone sequence of parameters $\lambda_n$ converging to $\lambda_*$, for which the corresponding vector field $X_{\lambda_n}$ has a tangency of $W^u(E)$ and $W^s(E)$ such that the homoclinic tangency orbit intersects with $T$ and has its lenght within $T$ tending to infinity as $n\to\infty$.
\end{definition}

\begin{definition}
    A family of vector fields $X_\lambda$ on $\mathbb{R}^4$ satisfying (H1)-(H4) is said to have a \textbf{cusp-transverse heteroclinic cycle} at $\lambda=\lambda_0$, if the conditions below hold:
    \begin{description}
        \item[\textrm{(C1)}] $X_{\lambda_0}$ has a saddle-node periodic orbit $\gamma_* \subset H^{-1}(0)$ which is symmetric under the involution $R$.
        \item[\textrm{(C2)}] The saddle-node periodic orbit $\gamma_*$ is generic and generically unfolded in $X_\lambda$ under the reversibilty with respect to $R$.
        \item[\textrm{(C3)}] $W^u(E)$ and $W^s(\gamma^*)$, as well as $W^u(\gamma^*)$ and $W^s(E)$, intersect transversely, where $W^u(\gamma^*)$ and $W^s(\gamma^*)$ stand for the stable and unstable sets of the non-hyperbolic periodic orbit $\gamma^*$.
    \end{description}
\end{definition}
\begin{remark}
    The notions of cocooning cascade of homoclinic tangencies and cusp-transverse heteroclinic cycle are adaptions of the notions of cocooning cascade of heteroclinic tangencies and cusp-transverse heteroclinic chain, respectively, introduced in \cite[Definitions 1.3 and 1.4]{dumibakok2006}. In that paper, one assumes that the stable and unstable sets of the saddle-node periodic orbit are intersected by two-dimensional invariant manifolds of two saddle-type equilibrium points with different stability indices. Instead of homoclinic, one has to deal with heteroclinic orbits.
\end{remark}
We prove the following result:
\begin{theorem}
\label{th:cocooning_cascade}
    Let $X_{\lambda}$ be a smooth family of reversible vector fields on $\mathbb{R}^4$ satisfying (H1)-(H4). Suppose that at $\lambda=\lambda_0$ the corresponding vector field $X_{\lambda_0}$ has a cusp-transverse heteroclinic cycle. Then the family $X_\lambda$ exhibits a cocooning cascade of homoclinic tangencies centered at $\lambda_0$.
\end{theorem}
\begin{proof}
A symmetric periodic orbit $\gamma$ intersects $\mathrm{Fix}(R)$ at exactly two points (see \cite{vanfie1992,har1998,lamrob1998}).
Let $p_0$ be one of the two points where $\gamma_*$ meets $\mathrm{Fix}(R)$. Due to condition (H4), in a neighborhood of $p_0$, the equation $H=0$ defines either $x_1$ or $x_3$ as a function of the other variables. Without loss of generality, we assume that $\frac{\partial H}{\partial x_3}(p_0)\neq 0$. Therefore, we can take a coordinate chart on $H^{-1}(0)$ with variables $(x_1,x_2,x_4)$ in a $R$-invariant neighborhood $V$ of $p_0$. That is, locally around $p_0$ we write
\[
H^{-1}(0)_{loc}=\{(x_1,x_2,x_3,x_4) \in V \,:\, (x_1,x_2,x_4) \in U,\, x_3=f(x_1,x_2,x_4)\},
\]
where $U$ is a three-dimensional domain invariant under the involution
\[
(x_1,x_2,x_4) \rightarrow (x_1,-x_2,-x_4)
\]
and $f:U \rightarrow \mathbb{R}$ is such that 
\[
H(x_1,x_2,f(x_1,x_2,x_4),x_4)=0
\]
and
\begin{equation}
\label{eq_reversibility_f}
f(x_1,-x_2,-x_4)=f(x_1,x_2,x_4).    
\end{equation}

Now, let $C^* \subset H^{-1}(0)$ be a cross section (with respect to the three-dimensional flow restricted to $H^{-1}(0)$), satisfying that $R(C^*)=C^*$. This cross section must contain $\mathrm{Fix}(R) \cap H^{-1}(0)_{loc}$. Since
\[
\mathrm{Fix}(R) \cap H^{-1}(0)_{loc}=\{(x_1,0,x_3,0) \in V \,:\, (x_1,0,0) \in U,\, x_3=f(x_1,0,0)\},
\]
we have to take a coordinate chart on $C^*$ containing variable $x_1$. Without loss of generality, we assume that a coordinate chart can be given with variables $(x_1,x_4)$; otherwise, we should take coordinates $(x_1,x_2)$, but the subsequent arguments would be analogous. For this reason, we write $C^*$ as 
\[
C^*=\{(x_1,x_2,x_3,x_4) \in V \,:\, (x_1,x_4) \in W,\, x_2=g(x_1,x_4),\,x_3=f(x_1,g(x_1,x_4),x_4)\},
\]
where $W$ is a two-dimensional domain invariant under the involution
\[
(x_1,x_4) \rightarrow (x_1,-x_4)
\]
and $g:W \rightarrow \mathbb{R}$ is such that 
\[
(x_1,g(x_1,x_4),x_4) \in U
\]
and
\begin{equation}
\label{eq_reversibility_g}
g(x_1,-x_4)=-g(x_1,x_4).
\end{equation}

Let $C_0^* \subset C^*$ be a subsection, also invariant with respect to $R$, such that the Poincare map
\[
\Pi^*: C_0^* \subset C^* \rightarrow C^*
\]
is well defined.
\begin{remark}
    In general, a periodic orbit will intersect $C$ at two or more different points. Suppose that it crosses $C$ at $m$ points. Therefore, $\Pi^*$ is given by the composition of a convenient restriction of the general Poincaré map $\Pi: C \rightarrow C$ with itself $m$ times.
\end{remark}

Given the involution
\[
\widehat{R}: (x_1,x_4) \rightarrow (x_1,-x_4),
\]
it easily follows that
\[
\widehat{R} \circ \Pi^*= (\Pi^*)^{-1} \circ \widehat{R}.
\]
Indeed, let $(x_1,g(x_1,x_4),f(x_1,g(x_1,x_4),x_4),x_4) \in C^*$ be any point where the Poincaré map is well-defined, that is, there exists $\bar t >0$ such that 
\[
(\bar x_1,\bar x_2,\bar x_3,\bar x_4)=\varphi(\bar t,(x_1,g(x_1,x_4),f(x_1,g(x_1,x_4),x_4),x_4)) \in C^*
\]
and $\varphi(t,(x_1,g(x_1,x_4),f(x_1,g(x_1,x_4),x_4),x_4)) \notin C^*$ for all $t\in ]0,\bar t[$. Then, $\Pi\left((x_1,x_4)\right)=(\bar x_1,\bar x_4)$. Due to the reversibility of \eqref{sistema_principal},
\begin{eqnarray*}
\lefteqn{\varphi(-\bar t,R((x_1,g(x_1,x_4),f(x_1,g(x_1,x_4),x_4),x_4)))}
\\
&=&
R(\varphi(\bar t,(x_1,g(x_1,x_4),f(x_1,g(x_1,x_4),x_4),x_4))),
\end{eqnarray*}
that is,
$$
\varphi(-\bar t,(x_1,-g(x_1,x_4),f(x_1,g(x_1,x_4),x_4),-x_4))=(\bar x_1,-\bar x_2,\bar x_3,-\bar x_4)\in C^*,
$$
and there is no $t \in ]-\bar t,0[$ such that $\varphi(-\bar t,(x_1,-g(x_1,x_4),f(x_1,g(x_1,x_4),x_4),-x_4))\in C^*$. Taking into account \eqref{eq_reversibility_f} and \eqref{eq_reversibility_g}, 
\[
(x_1,-g(x_1,x_4),f(x_1,g(x_1,x_4),x_4),-x_4)=
(x_1,g(x_1,-x_4),f(x_1,g(x_1,-x_4),-x_4),-x_4)
\]
Therefore, it follows that $(\Pi^*)^{-1}((x_1,-x_4))=(\bar x_1,-\bar x_4)$.

From here, the proof is analogous to the one of Theorem 1.5 in \cite{dumibakok2006}. Let $\Pi^*_\lambda:C_0^* \subset C^* \rightarrow C^*$ be the Poincaré map along $\gamma^*$ parameterized with respect to $\lambda$ and expressed as a function of variables $(x_1,x_4)$. We proved above that the the Poincaré map is reversible under the involution $\widehat R$. As argued in \cite[Remark 1.7]{dumibakok2006}, at the saddle-node point $p_0$, the differential $D\Pi^*_\lambda(p_0)$ must have a double eigenvalue $1$. Furthermore, in \cite{dumibakok2006} it is proven that, under the genericity condition (C2), $D\Pi^*_\lambda(p_0)$ is conjugate to the unipotent matrix
\[
\left(
\begin{array}{cc}
    1 & 1
    \\
    0 & 1
\end{array}
\right),
\]
and also there exists a stable branch $\gamma_s$ and an unstable branch $\gamma_u$ emanating from $p_0$ \cite[Theorem 2.4]{dumibakok2006}. Condition (C3) implies that $L^u=W^u(E) \cap C^*$ and $L^s=W^s(E) \cap C^*$ intersect transversely with the stable set $\gamma_s $ and the unstable set $\gamma_u$, as illustrated in Figure \ref{fig:saddle_node} (left). From the configuration depicted in that figure, the use of the terminology of cusp transverse heteroclinic cycle makes sense.
From \cite[Theorem 2.5]{dumibakok2006}, it follows that for $n$ large enough, the iterates $(\Pi^*)^n(L^u)$ and $(\Pi^*)^{-n}(L^s)$ intersect each other, as represented in Figure \ref{fig:saddle_node} (left).
\begin{figure}
\centering
\begin{tabular}{cc}
    \includegraphics[width=0.47\linewidth]{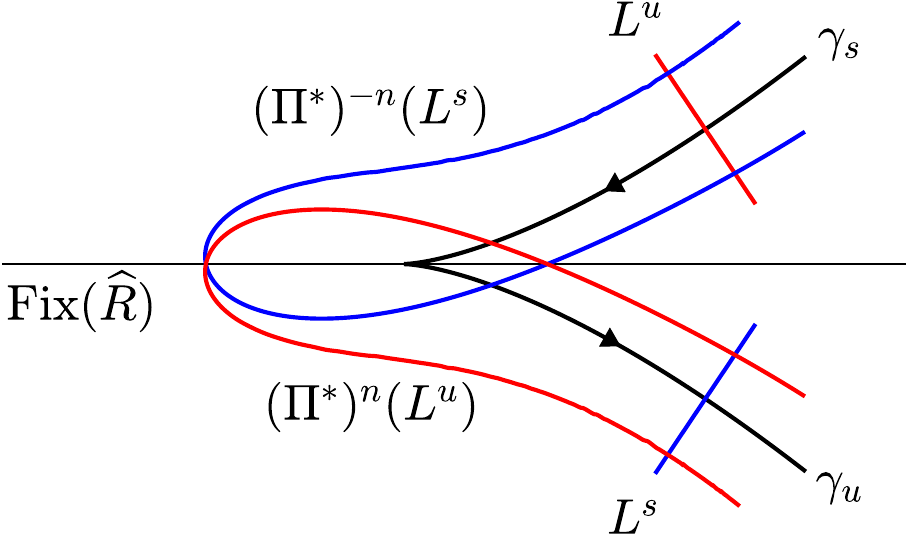}
    & 
    \includegraphics[width=0.43\linewidth]{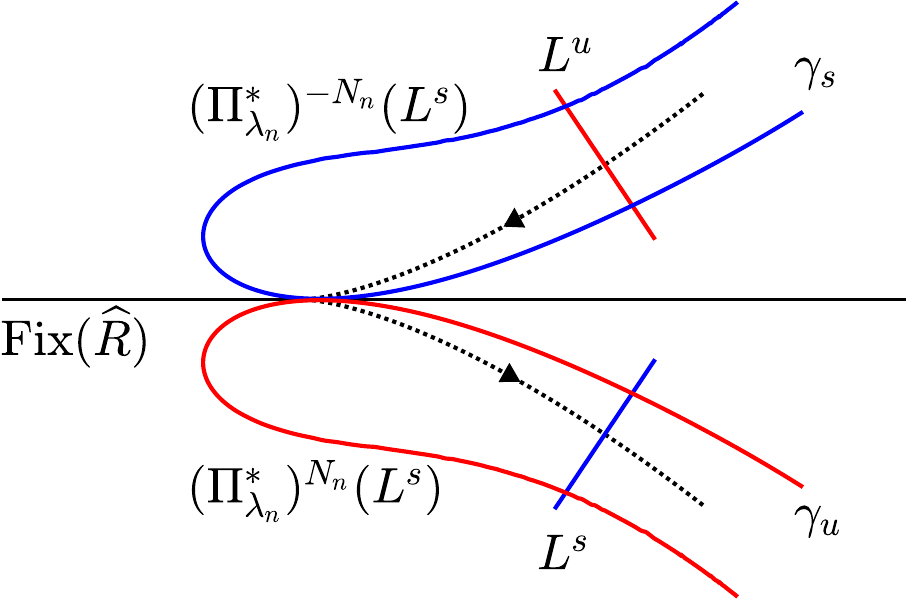}
\end{tabular}
    \caption{Left: A cusp transverse heteroclinic cycle is illustrated. The curves $\gamma_s$ and $\gamma_u$ correspond to the intersections of the stable and unstable sets of the saddle-node periodic orbit $\gamma^*$. The contact point between $\gamma_s$ and $\gamma_u$ corresponds to the fixed point given by the intersection between $\gamma^*$ and the Poincaré section $C^*$ (see the text for details). $L^u$ and $L^s$ are, respectively, intersections between the unstable and stable manifolds at the equilibrium point with $C^*$. As shown in the picture, $L^u$ and $L^s$ intersect transversely $\gamma_s$ and $\gamma_u$, respectively. As explained in the text, after an enough number of iterations of the Poincaré map and its inverse, the invariant manifolds of the equilibrium points must intersect each other transversely, as shown in the figure. Right: A homoclinic tangency appears when the fixed points disappear through the saddle-node bifurcation.}
    \label{fig:saddle_node}
\end{figure}

Now, we apply \cite[Theorem 2.6]{dumibakok2006} to conclude that there exist a sequence of parameters $\{\lambda_n\}$ that converge to $\lambda_0$, which can be chosen monotone, and a sequence of integers $\{N_n\}$ that diverge to $\infty$ such that $(\Pi_{\lambda_n}^*)^{N_n}(L^u)$ and $(\Pi_{\lambda_n}^*)^{-N_n}(L^s)$ have a point of tangency (right plot in Figure \ref{fig:saddle_node}). This means that for each $\lambda_n$, the corresponding vector field $X_{\lambda_n}$ has a homoclinic orbit $\Gamma_n$ along which $W^u(E)$ and $W^s(E)$ intersect tangentially. Moreover, since $N_n \to \infty$ as $n \to \infty$, the length of $\Gamma_n$ diverges inside any a priori given tubular neighborhood of the saddle-node periodic orbit $\gamma^*$. This completes the proof of Theorem \ref{th:cocooning_cascade}. 
\end{proof}

\subsection{Reversible saddle-node periodic bifurcations in the model.}
The study of the images of $\mathrm{Fix}(R)$ by successive iterations of the Poincaré map $\Pi$, provides numerical evidence of the existence of an infinite number of parameter values for which the system exhibits saddle-node bifurcations of periodic orbits. To illustrate this fact, we fix $\eta_3=-1.776$ and consider the initial conditions in $\mathrm{Fix}(R) \cap C_3$.  Iterations $\Pi^n(\mathrm{Fix}(R) \cap C_3)$, with $n=1,\ldots,11$ are given in the top panel in Figure \ref{fig:iteradas_del_fix}.
\begin{figure}
    \centering
    \includegraphics[width=0.84\linewidth]{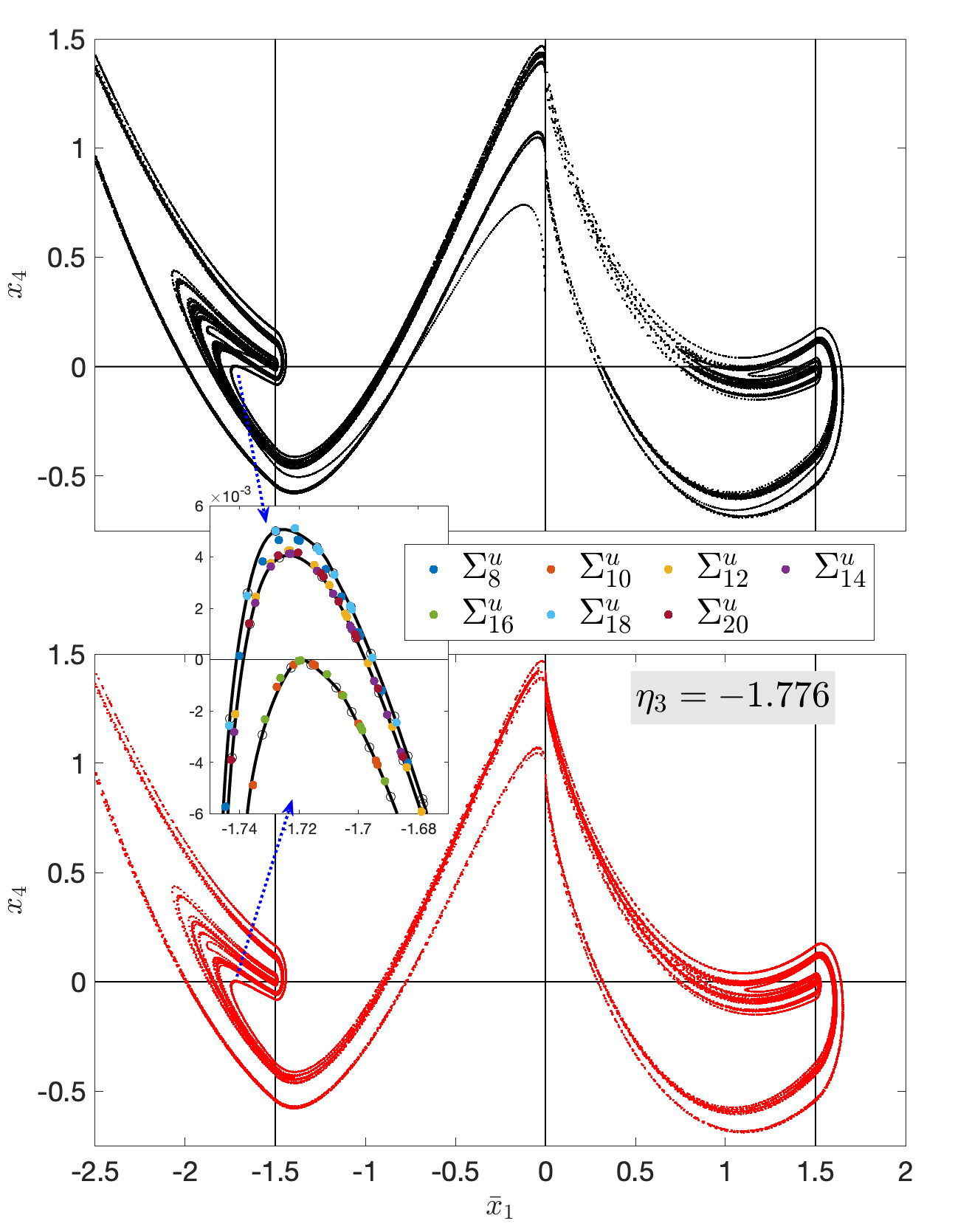}
    \caption{Top: Eleven iterations of the Poincaré map $\Pi: C \rightarrow C$ are shown, with initial points on $\mathrm{Fix} (R) \cap C_3$ ($0 < x_1 < 1.5$ and $x_3<0$). Bottom: Intersections $\Sigma_i^u$ are shown for $i=1,\ldots,20$. If these intersections were represented in the upper panel, an overlap with the iterations of $\mathrm{Fix}(R) \cap C_3$ would be clearly observed. In the central part of the figure, an enlargement of a small region in $C_1$ ($\bar x_1 \in ]-1.75,-1.67[$ and $x_4 \in ]-0.006,0.006[$) is shown. The blue arrows indicate approximately the area that is enlarged. Black curves correspond to pieces of $\Pi^{11}(\mathrm{Fix} (R) \cap C_3)$. They are obtained by computing splines through empty black circles (the numerically computed iterations). Colored points correspond to part of the intersections $\Sigma_i^u$ for $i=8,\ldots,20$ inside the enlargement.}
    \label{fig:iteradas_del_fix}
\end{figure}
To facilitate comparison, intersections $\Sigma^u_i$, where $i=1,\ldots,20$, are depicted in the bottom panel. Comparing top and bottom panels allows one to observe the overlapping nature of the curves. This observation is not surprising. For this parameter value, the primary homoclinic orbit is determined by the intersection between $\Sigma_3^u$ and $\Sigma_3^s$ at a point $q_{3,3} \in \mathrm{Fix} (R) \cap C_3$ (compare with the plot corresponding to $\eta_3=-1.84$ in Figure \ref{fig:cuatro_cortes}). Orbits of points in $\mathrm{Fix} (R) \cap C_3$ sufficiently close to $q_{3,3}$ will follow the homoclinic orbit and, after two iterations, will go through a local transition around the bifocus equilibrium point. After this passage, they depart from a neighborhood of the origin through points near $D^u$, the proximity increasing as the initial point approaches $q_{3,3}$. Consequently, subsequent iterations closely follow $\Sigma^u$.

It seems evident that, through the variation of the parameters, infinitely many points of tangency should emerge between the iterations of $\Pi^k(\mathrm{Fix}(\widehat R) \cap C_3)$ and $\mathrm{Fix}(\widehat R)$ itself. Each of these tangencies corresponds to a symmetric periodic orbit of saddle-node type. Figure \ref{fig:iteradas_del_fix} features an enlargement of a small square $[-1.75,-1.67] \times [-0.006,0.006]$ in the plane $(\bar x_1,x_4)$. The black curves represent three selected arcs contained in $\Pi^{11}(\mathrm{Fix}(\widehat R) \cap C_3)$ (constructed using numerical data interpolated by splines). The lowest curve closely approaches a tangency (in fact, a small variation of $\eta_3$ would result in such a tangency). The other two curves exhibit two intersections with $\mathrm{Fix}(\widehat R)$, presenting two periodic orbits with $22$ intersections with $C$. Colored points belong to $\Sigma^u$ and further illustrate the proximity of the iterations of $\mathrm{Fix}(\widehat R)$ to $\Sigma^u$.

Our numerical results provide the existence of saddle-node bifurcations of periodic orbits and also homoclinic tangency bifurcations. However, in order to obtain numerical evidence for the existence of cusp transverse heteroclinic cycles, a more accurate exploration is required (similar to that in \cite{dumibakok2006} or, with much more work, a computer-assisted argumentation like that in \cite{kokwilzgl2007}). Although this is a task for further research, a preliminary analysis of the linearization of the Poincaré map, around saddle-type hyperbolic orbits close to a saddle-node bifurcation, provides additional support for the conjecture.

\section{Discussion}
\label{sec:discussion}
Certainly, going beyond the results achieved and discussed in \cite{bufchatol1996} is challenging. The authors themselves highlight the intricacies involved in obtaining further advances. However, our paper introduces a novel approach to investigate bifocal homoclinic connections in \eqref{sistema_principal}. The focus lies in studying how the invariant manifolds of the bifocus intersect a specific cross section. Originally conceived as a three-dimensional section, it is effectively reduced to a two-dimensional surface considering only orbits within the $0$-level set of the first integral $H$. The resulting cross section becomes a loop cylinder, which can be represented on a plane after proper unfolding. Through these concepts, we achieve a visual understanding of the intricate geometry inherent in the invariant manifolds.
We believe that the numerical approximation of invariant manifolds and the exploration of their intersections with an appropriate cross section, as proposed in this study, have the potential to illuminate the structure of the set of homoclinic orbits. It should be noted that a similar perspective is outlined in \cite[Section 4.4]{bufchatol1996}, albeit without delving into details, as the focus was mainly on numerical continuation methods.

We have explained the parameterization method applied for numerically approximating invariant manifolds with high precision. Our approach also accounts for the unique perspective we employ in examining intersections with a cross section. It produces informative visualizations that offer insights into the creation of homoclinic orbits. With these elements, we provide a preliminary discussion on the transformations within the homoclinic structure as the parameter changes.

A logical extension of our findings would involve delving into the conjectures presented in \cite{bufchatol1996}. In particular, \cite[Section 5]{bufchatol1996} introduces a set of rules for the coalescence of symmetric homoclinic orbits. They should be explored using our methods. Establishing a connection between our proposal for homoclinic orbit labeling and that in \cite{bufchatol1996} becomes crucial for this study. Perhaps a refinement of our labeling method may be necessary.

Furthermore, our methods have allowed us to place, within the framework of the system \eqref{sistema_principal}, the concepts of cocooning cascades of homoclinic tangencies and cusp-transverse heteroclinic cycles, drawing on the results of \cite{dumibakok2006}. This leads to conjecture the existence of these objects in the model. To support this conjecture, evidences are provided for the presence of saddle-node bifurcations of periodic orbits.

Looking ahead, a more detailed exploration of these bifurcations using continuation methods and actively seeking the simplest bifurcations with the minimum number of intersections with the cross section, holds substantial interest. For such bifurcations, a computer-assisted proof, similar to the approach in \cite{kokwilzgl2007}, could be a valuable option to rigorously establish the existence of cusp-transverse heteroclinic cycles.

Explorations in Section \ref{sec:intersections} suggest an intriguing problem for study: finding the limit of $\cup_{i=1}^{\infty} \Sigma_i ^u$ as $\eta_3$ tends to $-2$. The natural approach is to examine the Hamiltonian Hopf bifurcation that occurs at $\eta_3=-2$ (as discussed, for instance, in \cite{gaigel2011,iooper1993}). In particular, resonance is absent for values of $\eta_3 < -2$. In this scenario, the bifurcation at the origin takes the form of a Hopf-Hopf bifurcation of codimension $2$, and its topological type will need to be determined following the classical classifications presented in \cite{guchol2013} or \cite{kuznetsov2023}.

Of course, our main motivation lies in the study of the entire limit family \eqref{eq:limit_family}. A crucial initial step involves the exploration of two-parametric bifurcation diagrams, keeping $\eta_3$ constant while varying $\eta_2$ and $\eta_4$ within a neighborhood of $0$. Although $]-2,2[$ is a very interesting range, the region $\eta_3 \leq -2$ is equally significant, particularly given the insights gained from the study of Hopf-Hopf singularities.

\section*{Acknowledgements}
Authors has been supported by the Spanish Research project  PID2020-113052GB-I00.

\bibliographystyle{plainnat}

\end{document}